\documentclass[12pt]{article}
\usepackage{amsthm,amscd}
 \usepackage{txfonts}

\usepackage{amsmath}
\usepackage{mathrsfs}
\usepackage{enumerate}
\newenvironment{tightenumerate}{
\begin{enumerate}[(1)]
  \setlength{\itemsep}{3pt}
  \setlength{\parskip}{0pt}
}{\end{enumerate}}
\usepackage[english]{babel}
\usepackage{color}
\usepackage{hyperref}
\usepackage{lastpage}
\usepackage{fancyhdr}
\usepackage{geometry}
\usepackage{graphicx}
\hypersetup{
    colorlinks,
    citecolor=black,
    filecolor=black,
    linkcolor=black,
    urlcolor=black
}
\usepackage{fancyhdr}
\usepackage{tocloft}

\usepackage{mathtools}
\mathtoolsset{showonlyrefs}
\allowdisplaybreaks[1]

\numberwithin{equation}{section}
\theoremstyle{plain}
	\newtheorem{theorem}{Theorem}[section]           
	\newtheorem{corollary}[theorem]{Corollary}             
	\newtheorem{lemma}[theorem]{Lemma}                 
	\newtheorem{proposition}[theorem]{Proposition}          
	\newtheorem{claim}[theorem]{Claim}				
	\newtheorem{assumption}{Assumption}[section]            
\theoremstyle{definition}
\theoremstyle{remark}
	\newtheorem{remark}[theorem]{Remark}                

\newcommand\bR{\mathbf{R}}
\newcommand\bE{\mathbf{E}}

\newcommand\bP{\mathbf{P}}

\newcommand\bN{\mathbf{N}}

\newcommand\cB{\mathcal{B}}
\newcommand\cC{\mathcal{C}}

\newcommand\cF{\mathcal{F}}

\newcommand\cO{\mathcal{O}}

\newcommand\cX{\mathcal{X}}

\title{\vspace{-1.2in}On Some Properties of Space Inverses of Stochastic Flows  \vspace{-0.5cm}}

\begin{document}

\maketitle

\begin{tabular}{l}
\thispagestyle{fancy}
{\large\textbf{James-Michael Leahy}} \vspace{0.2cm}\\
 The University of Edinburgh, E-mail: \href{J.Leahy-2@sms.ed.ac.uk}{J.Leahy-2@sms.ed.ac.uk}\vspace{0.4cm} \\
{\large \textbf{Remigijus Mikulevi\v{c}ius}}\vspace{0.2cm}\\
 The University of Southern California, E-mail: \href{Mikulvcs@math.usc.edu }{mikulvcs@math.usc.edu}\vspace{0.4cm}\\
 {\large\textbf{Abstract}} \vspace{0.2cm} \\
 \begin{minipage}[t]{0.9\columnwidth}%
We derive moment estimates and a strong limit theorem for  space inverses of stochastic flows generated by jump SDEs with adapted coefficients in weighted H{\"o}lder norms using the Sobolev embedding theorem and the change of variable formula.  As an application of some basic properties of flows of continuous SDEs, we derive the existence and uniqueness of classical solutions of linear parabolic second order SPDEs by partitioning  the time interval and passing to the limit.  The methods we use allow us to improve on previously known results in the continuous case and to derive new ones in the jump case.  
 \end{minipage}
\end{tabular}

\tableofcontents
\thispagestyle{empty}

\section{Introduction}

 Let $\left( \Omega ,\mathcal{F}, \mathbf{F}=(\cF_t)_{t\ge 0},\bP\right) $ be a complete filtered
probability space satisfying the usual conditions of right-continuity and completeness.  Let $(w_{t}^{\varrho})_{\rho \ge 1}$,  $t\ge 0$, $\varrho\in\bN$, be a sequence of  independent one-dimensional $\mathbf{F}$-adapted Wiener processes.  For a $(Z,\mathcal{Z},\pi)$ is a
sigma-finite measure space, we let  $p(dt,dz)$  be an $\mathbf{F}$-adapted Poisson random measure on $(\bR_+
\times Z,\mathcal{B}(\bR_+)\otimes \mathcal{Z})$ with intensity measure $\pi(dz)dt$ and denote by
$
q(dt,dz )=p(dt,dz)-\pi(dz)dt
$
the compensated Poisson random measure.
For each real number $T>0$, we let  $\mathcal{R}_T$  and  $\mathcal{P}%
_T $ be the $\mathbf{F}$-progressive and $\mathbf{F}$-predictable
sigma-algebra on $\Omega\times [0,T] $, respectively.

Fix a real number $T>0$ and an integer $d\geq 1$. For each stopping time $\tau\le T$, consider the stochastic flow $X_t=X_{t}(\tau,x)$, $(t,x)\in [0,T]\times \mathbf{R}^d$,  generated by the stochastic differential equation (SDE) 
\begin{align}  \label{eq:SDEIntro}
dX_{t}&=b_t(X_{t})dt+\sigma^{\varrho}_t(X_{t})dw^{\varrho}_{t}+\int_{Z}H_t(X_{t-},z )q(dt,dz), \;\tau <t\le T,\notag\\
X_{t} &=x,\; t \leq \tau,
\end{align}
where  $b_t(x)=(b^{i}_t(\omega,x))_{1\leq i\leq d}$ and $\sigma_t (x)=(\sigma_t ^{i\varrho }(\omega,x)
_{1\leq i\leq d,\rho\ge 1}$ are $\mathcal{R}_T\otimes \mathcal{B}(%
\mathbf{R}^{d})$-measurable random fields defined on $\Omega \times \lbrack
0,T]\times \mathbf{R}^{d}$ and $H_t(x,z)=(H^{i}_t(\omega,x,z))_{1\leq i\leq d}$ is
a $\mathcal{P}_{T}\otimes \mathcal{B}(\mathbf{R}^{d})\otimes \mathcal{Z}$%
-measurable random fields defined on $\Omega \times [ 0,T]\times \mathbf{R}%
^{d}\times Z.$  The summation convention with respect to the repeated index  $\varrho\in \bN$ is used here and below.  In this paper, under natural regularity assumptions on the coefficients $b$, $\sigma$, and $H$,  we provide a simple and  direct derivation of moment estimates  of  the space inverse of the flow, denoted $X_t^{-1}(\tau,x)$,  in weighted H{\"o}lder norms by  applying  the Sobolev embedding theorem and the change of variable formula. Using a similar method, we  establish a strong limit theorem in weighted H{\"o}lder norms  for a sequence of flows $X_t^{(n)}(\tau,x)$ and their inverses $X_t^{(n);-1}(\tau,x)$ corresponding to  a sequence of coefficients $(b^{(n)},\sigma^{(n)},H^{(n)})$ converging in an appropriate sense.  Furthermore, as an application of the diffeomorphism property of flow,  we give a direct  derivation of the linear second order degenerate stochastic partial differential equation  (SPDE) governing the inverse flow $X_t^{-1}(\tau,x)$ when $H\equiv 0$. Specifically, for each $\tau\le T$, consider the stochastic flow  $Y_t=Y_t(\tau,x)$,
 $(t,x) \in [0,T]\times \mathbf{R}^d$, generated by the SDE
 \begin{align*}
 dY_{t} &=b_t(Y_t)dt+\sigma^{\varrho}_t (Y_{t}) dw^{\varrho}_{t}, \;\;\tau <t\le T,\\
 Y_{t} &=x,\;\;t\leq \tau.\notag
 \end{align*}
Assume that $b$ and $\sigma$ have  linear growth,  bounded first and second derivatives, and that  the second derivatives of $b$ and $\sigma$ are  $\alpha$-H{\"o}lder for some $\alpha>0$.  By partitioning the time interval  and using Taylor's theorem, the Sobolev embedding theorem, and some basic  properties of the flow and its inverse, we show that $u_t(x)=u_t(\tau,x):=Y_t^{-1}(\tau,x)$, $(t,x)\in [0,T]\times \mathbf{R}^d$ is the unique classical solution of the  SPDE  given by
\begin{align}\label{eq:SPDEIntro}
du_t (x)&=\left(\frac{1}{2}\sigma
^{i\varrho}_t(x) \sigma^{j\varrho}_t(x)\partial_{ij}u_t(x)-\hat b^i_t(x)\partial_iu_t(x)\right)d t -\sigma^{i\varrho}_t(x) \partial_iu_{t}(x)dw^{\varrho}_t,\;\;\tau <t\le T,\notag \\
u_t(x) &=x,\;\;t\leq \tau,
\end{align}
where 
$$
\hat b^i_t(x)=b^i_t(x)-\sigma^{j\varrho}_t(x) \partial_j\sigma^{i\varrho}_t(x).$$
In \cite{LeMi14}, we use all of the properties of the flow $X_t(\tau,x)$ that are established in this  work in order to derive  the existence and uniqueness  of classical solutions of linear parabolic stochastic integro-differential equations (SIDEs).\\
\indent One of the earliest works to investigate  the homeomorphism property of  flows of SDEs  with jumps is by P. Meyer in \cite{Me81a}.  In \cite{Mi83}, R. Mikulevi\v{c}ius extended the properties found in \cite{Me81a} to SDEs driven by arbitrary continuous martingales and random measures.   Many other authors have since expanded upon the work in \cite{Me81a}, see for example  \cite{FuKu85, Ku04,Me07, QiZh08,  Zh13b, Pr14} and references therein.  In \cite{Ku04,Ku86a}, H. Kunita studied the diffeomorphism property  of the flow $X_t(s,x),$ $(s,t,x)\in [0,T]^2\times\mathbf{R}^d$, and  in the setting of  deterministic coefficients,  he showed that for each fixed $t$, the  inverse  flow $X_t^{-1}(s,x),$ $(s,x)\in [t,T]\times\mathbf{R}^d$, solves a backward SDE.  By estimating the associated backward SDE,  one can obtain moment estimates  and a strong limit theorem for the inverse flow in essentially the same way that moment estimates are obtained for the direct flow (see, e.g. \cite{Ku86a}).  However, this method of deriving moment estimates and a strong limit theorem for the inverse flow uses a time reversal, and thus requires that the coefficients are deterministic.  In the case $H\equiv 0$, numerous authors have investigated properties of  the inverse flow with random coefficients.  In Chapter 2 of \cite{Bi81}, Lemma 2.1 and 2.2. of \cite{OcPa89},  and Section 6.1 and 6.2 of  \cite{Ku96}, the authors  derive properties of $Y^{-1}_t(\tau,x)$  (e.g. moment estimates, strong limit theorem, and the fact that it solves \eqref{eq:SPDEIntro})   by first  showing that it  solves the Stratonovich form SDE for $Z_t=Z_t(\tau,x)$, $(t,x)\in [0,T]\times \mathbf{R}^d$, given by
\begin{align}\label{eq:invflowsde}
dZ_t(x)&=- U_t(Z_t(x)) b_t(x)dt- U_t(Z_t(x))\sigma ^{\varrho}_t(x)\circ dw^{\varrho}_t,\;\;\tau <t\le T,\\
Z_0(x)&=x, \;\;\tau <t,\notag
\end{align}
where  $U_t(x)=U_t(\tau,x)=[\nabla Y_t(\tau,x)]^{-1}$.   In order to obtain a strong solution to \eqref{eq:invflowsde}, the authors impose conditions on the coefficients that guarantee  $\nabla U_t(x)$ is   locally-Lipschitz in $x$. In the degenerate setting,  the third derivative of   $b_t$ and $\sigma_t$  need to be $\alpha$-H{\"o}lder for some $\alpha>0$  to obtain that $\nabla U_t(x)$ is  locally-Lipschitz in $x$. However, for some reason, the authors assume more regularity than this.  In this paper, we derive properties of the inverse flow under those assumptions which guarantee that $Y_t(\tau,x)$ is a $\cC^{\beta}_{loc}$-diffeomorphism (and with $\beta>1$).

Classical solutions of \eqref{eq:SPDEIntro} have been constructed in \cite{Bi81,Ku96} by directly showing that  $Y^{-1}_t(\tau,x)$ solves \eqref{eq:invflowsde}. As we have mentioned above, this approach requires  the third derivatives of $b_t$ and $\sigma_t$  to be $\alpha$-H{\"o}lder for some $\alpha>0$. Yet another approach to deriving existence of classical solutions of  \eqref{eq:SPDEIntro} is using the method of time reversal (see, e.g. \cite{Ku96,DaTu98}). While this method  only requires that  the  second derivatives of $b_t$ and $\sigma_t$  are $\alpha$-H{\"o}lder for some $\alpha>0$, it does impose that the coefficients are deterministic.  In  \cite{KrRo82a},  N.V. Krylov and B.L. Rozvskii derived the existence and uniqueness of  generalized solutions of degenerate second order linear parabolic SPDEs  in Sobolev spaces  using variational approach of SPDEs and the method of vanishing viscosity (see, also, \cite{GeGyKr14} and Ch. 4, Sec. 2, Theorem 1 in \cite{Ro90}).  Thus, by appealing to the Sobolev embedding theorem, this  theory can be used to obtain classical solutions of degenerate linear SPDEs.  Proposition 1 of Ch. 5, Sec. 2 ,in \cite{Ro90} shows that if $\sigma$ is uniformly bounded and four-times continuously differentiable  in $x$ with uniformly bounded derivatives and $b$ is uniformly bounded and  three-times continuously differentiable with uniformly bounded derivatives, then there exists a classical solution of \eqref{eq:SPDEIntro} and $u_t(x)=Y_t^{-1}(x)$. This  is more regularity than we require.  

This paper is organized as follows. In Section 2, we state our  notation and the main results.  Section 3 is devoted to the proof of the properties of the stochastic flow  $X_t(\tau,x)$ and Section 4 to the proof that $Y^{-1}_t(\tau,x)$ is the unique  classical  solution of \eqref{eq:SPDEIntro}. In Section 5, the appendix,
auxiliary facts that are used throughout the paper are discussed.

\section{Outline of main results}

For each integer $n\ge 1$, let $\mathbf{R}^{n}$  be the $n$%
-dimensional Euclidean space and for each $x\in\mathbf{R}^{n}$, denote by $|x|$ the Euclidean norm of $x$. Let $\bR_+$ denote the set of non-negative real-numbers. Let $\bN$ be the set of natural numbers. 
Elements of $\mathbf{R}^d$ are  understood as column vectors  and elements of $\mathbf{R}^{2d}$ are understood as matrices of dimension $d\times d$.  We  denote the transpose of an element  $x\in\mathbf{R}^d$ by $x^*$. The norm of an element $x$ of  $\ell_2(\mathbf{R}^d)$ (resp. $\ell_2(\mathbf{R}^{2d})$), the space of square-summable $\mathbf{R}^d$-valued (resp. $\mathbf{R}^{2d}$-valued) sequences,  is  also denoted by $|x|$. For a topological space $(X,\cX)$ we denote the Borel sigma-field on $X$ by $\cB(X)$.

For each $i\in \{1,\ldots,d_1\}$, let $\partial_i=\frac{\partial}{\partial x_i}$ be
the spatial derivative operator with respect to $x_i$ and write  $\partial_{ij}=\partial_i\partial_j$ for each $i,j\in \{1,\ldots,d_1\}$.  For a once differentiable function $%
f=(f^1\ldots,f^{d_1}):\bR^{d_1}\rightarrow\bR^{d_1}$, we denote the gradient  of $f$ by $\nabla
f=(\partial_jf^i)_{1\le i,j\le d_1}$.  Similarly, for a once differentiable function $f=(f^{1\varrho},\ldots,f^{d\varrho})_{\varrho\ge 1} : \bR^{d_1}\rightarrow \ell_2(\bR^{d_1})$, we denote the gradient of $f$ by $\nabla f=(\partial_jf^{i\varrho})_{1\le i,j\le d_1,\varrho\ge 1} $ and understand it as a function from $\bR^{d_1}$ to $\ell_2(\mathbf{R}^{2d_1})$.
For a multi-index $%
\gamma=(\gamma_1,\ldots,\gamma_d)\in\{0,1,2,\ldots,\}^{d_1}$ of length $%
|\gamma|:=\gamma_1+\cdots+\gamma_d$, denote by $\partial^{\gamma}$ the
operator $\partial^\gamma=\partial_1^{\gamma_1}\cdots \partial_d^{\gamma_d}$, where $\partial_i^0$ is the identity operator for all $i\in\{1,\ldots,d_1\}$. For each integer $d\ge 1$, we denote by $C_c^{\infty}(\bR^{d_1}; \bR^{d})$ the space of infinitely differentiable functions with compact support in $\bR^{d}$. 

\indent For a Banach space $V$ with norm $|\cdot |_{V}$, domain $Q$ of $%
\mathbf{R}^{d}$, and continuous function $f:Q\rightarrow V$,  we define 
$$
|f|_{0;Q;V}=\sup_{x\in Q}|f(x)|
$$
and
$$
[f]_{\beta;Q;V}=\sup_{x,y\in Q,x\neq y}\frac{|
f(x)-f(y)|_{V}}{|x-y|_{V}^{\beta }},\;\;\beta \in (0,1].$$
For each real number $\beta\in \mathbf{R}$, we write  $\beta =[\beta]^-+\{\beta\}^+$, and $\{\beta\}^+\in (0,1]$.   For a Banach space $V$ with norm $|\cdot |_{V}$,  real number $\beta>0$, and domain $Q$ of $%
\mathbf{R}^{d}$, we denote by $\cC^{\beta }(Q;V)$ the Banach space of
all bounded continuous functions $f:Q\rightarrow V$ having finite norm 
\begin{equation*}
|f|_{\beta ;Q;V}:=\sum_{| \gamma |\leq [\beta ]^-
}|\partial^{\gamma }f|_{0;Q;V}+\sum_{|\gamma|=[\beta]^-}[\partial^{\gamma}f]_{\{\beta\}^+ ;Q;V}.
\end{equation*}%
When $Q=\mathbf{R}^{d}$ and $V=\mathbf{R}^n$ or $V=\ell_2(\bR^n)$ for any integer $n\ge 1$, we drop the subscripts $Q$  and $V$ from the norm  $| \cdot |_{\beta;Q;V}$ and write $|\cdot |_{\beta}
$.   For a Banach space $V$ and for each $\beta>0$, denote by  $%
\cC_{loc}^{\beta}(\bR^d;V)$   the Fr\'echet space   of   continuous functions   $f:\mathbf{R}^d\rightarrow V$ satisfying $f\in \cC^{\beta}(Q;V)$ for all
bounded domains $Q\subset \mathbf{R}^{d}$. We call a function $f:\mathbf{R}%
^{d}\rightarrow \mathbf{R}^{d} $ a $\cC_{loc}^{\beta}(\bR^d;\bR^d)$-diffeomorphism  if $f$
is a homeomorphism and both $f$ and its inverse $f^{-1}$ are in $\cC_{loc}^{\beta}(\bR^d;\bR^d)$.\\
\indent For a Fr\'echet space $\chi$, we denote by $D([0,T];\chi)$ the space of $\chi$%
-valued c\`{a}dl\`{a}g functions on $[0,T]$ and by $C([0,T]^{2};\chi)$ the
space of $\chi$-valued continuous functions on $[0,T]\times \lbrack 0,T]$.  The spaces   $D([0,T];\chi)$  and $C([0,T]^{2};\chi )$ are endowed with the supremum semi-norms.\\  
\indent The notation $N=N(\cdot ,\cdots,\cdot )$ is used to denote a positive
constant depending only on the quantities appearing in the parentheses. In a
given context, the same letter is often used to denote different constants
depending on the same parameter. If we do not specify to which space  the parameters $%
\omega ,t,x,y,z$ and $n$ belong, then we  mean $\omega \in \Omega $, $%
t\in [ 0,T]$, $x,y\in \mathbf{R}^{d}$, $z\in Z$, and $n\in\mathbf{N}$.\\

\indent Let  $r_1(x)=\sqrt{1+|x|^2},$ $x\in\mathbf{R}^d$. For each real number $\beta > 1$, we introduce the following regularity condition on the coefficients $b,\sigma, $ and $H$.

\begin{assumption} [$\beta$]\label{asm:regularitypropflow}

\begin{tightenumerate}
 \item  There is a constant $N_{0}>0$ such
that for all $(\omega,t,z)\in \Omega\times [0,T]\times Z$,
$$
|r_1^{-1}b_t|_0+|\nabla b_t|_{\beta -1}+|r_1^{-1}\sigma_t|_0+|\nabla \sigma_t|_{\beta -1}\leq
N_{0}\quad \textit{and}\quad| r_1^{-1}H_t(z )|_{0}+|\nabla H_t(z )|_{\beta -1}\leq K_{t}(z
),
$$
where $K :\Omega \times[ 0,T]\times Z\rightarrow \mathbf{R}_+$ is a  $\mathcal{P}_{T}\otimes \mathcal{Z}$-measurable function satisfying
$$
K_{t}(z)+\int_{Z}K_t(z)^{2}\pi (dz)\leq
N_{0},
$$
for all
$(\omega,t,z)\in \Omega\times [0,T]\times Z$.
\item There are constants $\eta\in (0,1)$  and $N_{\kappa}>0$ such that for all $(\omega ,t,x, z)\in \{(\omega ,t,x,z)\in \Omega \times
[ 0,T]\times \mathbf{R}^d\times Z:|\nabla H_{t}(\omega ,x,z)|>\eta \},$ 
$$
|\left( I_{d}+\nabla
H_t(x,z)\right) ^{-1}|\leq N_{\kappa}.
$$
\end{tightenumerate}
\end{assumption}

The following theorem shows that if  Assumption \ref{asm:regularitypropflow} $(\beta)$  holds for some $\beta>1$, then for any $\beta'\in [1,\beta]$, the solution   $X_t(\tau,x)$ of \eqref{eq:SDEIntro} has a modification that  is a $\cC_{loc}^{\beta'}(\bR^d\bR^d)$\allowbreak-diffeomorphism  and  the $p$-th moments of the weighted $\beta '$-H{\"o}lder norms of the inverse flow are bounded. This theorem will be proved in the next section.

\begin{theorem}\label{thm:diffeoandmomest} Let  Assumption \ref{asm:regularitypropflow}$(\beta)$  hold for some $\beta>1$.
\begin{tightenumerate}
\item  For each stopping time $\tau\le T$ and $\beta'\in [1,\beta)$, there exists a modification of the strong solution $X_t(\tau,x)$ of \eqref{eq:SDEIntro}, also denoted by $X_t(\tau,x)$, such that  $\bP$-a.s.\ the mapping $X_{t}(\tau,\cdot )\allowbreak:\mathbf{R}%
^{d}\rightarrow \mathbf{R}^{d}$ is a $\cC_{loc}^{\beta'}(\bR^d;\bR^d)$%
-diffeomorphism,  $X_{\cdot}(\tau,\cdot),X^{-1}_{\cdot}(\tau,\cdot)\in D([0,T];\cC_{loc}^{\beta'}(\bR^d;\bR^d))$, and $X_{t-}^{-1}(\tau,\cdot )$ coincides
with the inverse of $X_{t-}(\tau,\cdot)$. Moreover, for each $\epsilon >0$ and $p\ge 2$, there is a constant $N=N(d,p,N_{0},T,\beta',\epsilon)$
such that  
\begin{equation}\label{eq:MomEstDirect}
\bE\left[\sup_{t\leq T}|r_{1}^{-(1+\epsilon )}X_{t}(\tau )|_{0}^{p}\right]+\bE\left[\sup_{t\leq T}|r_{1}^{-\epsilon }\nabla X_{t}(\tau )|_{\beta'-1}^{p}\right]\leq N
\end{equation}
and  a constant $N=N(d,p,N_{0},T,\beta',\eta ,N_{\kappa},\epsilon)$ such that
\begin{equation}\label{ineq:MomEstInverse}
\bE\left[\sup_{t\leq T}|r_{1}^{-(1+\epsilon )}X^{-1}_{t}(\tau )|_{0}^{p}\right]+\bE\left[\sup_{t\leq T}|r_{1}^{-\epsilon }\nabla X^{-1}_{t}(\tau )|_{\beta'-1}^{p}\right]\leq N.
\end{equation}
\item If $H\equiv 0$, then for each  $\beta'\in (1,\beta)$, $\bP$-a.s.\ $X_{\cdot}(\cdot ,\cdot),X^{-1}_{\cdot}(\cdot,\cdot)\in C([0,T]^2;\cC_{loc}^{\beta'}(\bR^d;\bR^d))$ and  for each $\epsilon >0$ and $p\ge 2$, there is a constant $N=N(d,p,N_{0},T,\beta',\epsilon)$
such that  
\begin{equation}\label{ineq:MomEstDirectcts}
\bE\left[\sup_{s,t\leq T}|r_{1}^{-(1+\epsilon )}X_{t}(s )|_{0}^{p}\right]+\bE\left[\sup_{s,t\leq T}|r_{1}^{-\epsilon }\nabla X_{t}(s )|_{\beta'-1}^{p}\right]\leq N
\end{equation}
and
\begin{equation}\label{ineq:MomEstInversects}
\bE\left[\sup_{s,t\leq T}|r_{1}^{-(1+\epsilon )}X^{-1}_{t}(s )|_{0}^{p}\right]+\bE\left[\sup_{s,t\leq T}|r_{1}^{-\epsilon }\nabla X^{-1}_{t}(s )|_{\beta'-1}^{p}\right]\leq N.
\end{equation}
\end{tightenumerate}
\end{theorem}
\begin{remark}
The estimate \eqref{ineq:MomEstInverse} is used in \cite{LeMi14} to  take the optional projection of a linear transformation of the inverse flow of a jump SDE driven by two independent Weiner processes and two independent Poisson random measures relative to the filtration generated by one of the Weiner processes and Poisson random measures.
\end{remark}

Now, let us  state our strong limit theorem for a sequence of flows, which will also be proved in the next section. We will use this strong limit theorem in \cite{LeMi14} to show that the inverse flow of  a jump SDE solves a parabolic stochastic integro-differential equation. For each $n$, consider the
 stochastic flow $X^{(n)}_t=X^{(n)}_t(\tau, x)$, $(t,x)\in[0,T]\times \mathbf{R}^d$, generated by the SDE
\begin{align*}
dX^{(n)}_t&=b^{(n)}_t(X^{(n)}_t)dt+\sigma^{(n)l\varrho}_t
(X ^{(n)}_t)dw^{\varrho}_{t}+\int_{Z}H^{(n)}_t(X^{(n)}_{t-},z )q(dt,dz ),\;\;\tau \leq
t\leq T, \\
X^{(n)}_t &=x,\;\;t\leq \tau.
\end{align*}
Here we assume that for each $n$, $b^{(n)}$, $\sigma^{(n)}$, and $H^{(n)}$  satisfy the same measurability conditions as $b,\sigma,$ and $H$, respectively.

\begin{theorem}\label{thm:stronglimit}Let  Assumption \ref{asm:regularitypropflow}$(\beta)$  hold for some $\beta>1$ and assume that  $b^{(n)}, \sigma^{(n)}$, and $H^{(n)}$  satisfy  Assumption \ref{asm:regularitypropflow} $(\beta)$  uniformly in $n\in \bN$. Moreover, assume that 
$$
d\bP dt-\lim_{n\rightarrow\infty}\left (|r_1^{-1}b^{(n)}_t- r_1^{-1}b_t|_{0}+|\nabla b^{(n)}_t-\nabla b_t|_{\beta-1}\right)=0,
$$
$$
d\bP dt-\lim_{n\rightarrow\infty}\left (|r_1^{-1}\sigma^{(n)}_t-r_1^{-1}\sigma_t|_{\beta-1}+|\nabla \sigma^{(n)}_t-\nabla \sigma _t|_{0}\right)=0,
$$
and  for all  $(\omega,t,z)\in \Omega\times [0,T]\times Z$ and $n\in \bN $,
$$
|r_1^{-1}H^{(n)}_t(z)-r_1^{-1}H_t(z)|_{0}+|\nabla H^{(n)}_t(z)-\nabla H_t(z)|_{\beta-1}\le K^{(n)}(t,z),
$$
where $(K^{(n)}_t(z))_{n\in\mathbf{N}}$ is a sequence of  $\mathbf{R}_{+}$-valued $\mathcal{P}_{T}\otimes \mathcal{Z}$ measurable functions
defined on $\Omega \times [ 0,T]\times Z$ satisfying for all $(\omega,t,z)\in \Omega\times [0,T]\times Z$ and $n\in \bN$,
$$
K^{(n)}_t(z)+\int_Z K_t^{(n)}(z)^2\pi(dz)\le N_0
$$
and 
$$
d\bP dt-\lim_{n\rightarrow\infty} \int_{Z} K_t^{(n)}(z)^2\pi(dz)=0.
$$ Then for each stopping time $\tau\le T$,  $\beta'\in [1,\beta)$,  $\epsilon>0,$ and $p\ge2,$  we have
\begin{gather*}
\lim_{n\rightarrow \infty }\left(\bE\left[\sup_{t\leq T} | r_1^{-(1+\epsilon)}X^{(n)}_t(\tau )-r_1^{-(1+\epsilon)}X_t(\tau)|_{0}^{p}\right]+\bE\left[\sup_{t\leq T} |
r_1^{-\epsilon}\nabla X^{(n)}_t (\tau) -r_1^{-\epsilon}\nabla X_{t}(\tau ) |
_{\beta'-1}^{p}\right]\right)=0,\\
\lim_{n\rightarrow \infty }\bE\left[\sup_{t\leq T} | r_1^{-(1+\epsilon)}X^{(n);-1}_t(\tau )-r_1^{-(1+\epsilon)}X^{-1}_t(\tau)|_{0}^{p}\right]=0,
\end{gather*}
and
$$
\lim_{n\rightarrow \infty }\bE\left[\sup_{t\leq T} |
r_1^{-\epsilon}\nabla X^{(n);-1}_t (\tau) -r_1^{-\epsilon}\nabla X^{-1}_{t}(\tau ) |
_{\beta'-1}^{p}\right]=0.
$$
\end{theorem}

Let us  introduce our class of solutions for the equation \eqref{eq:SDEIntro}.
For a each number $\beta' >2$, let $\mathfrak{C}_{cts}^{\beta'}(\bR^d;\bR^d)$  be the linear space of  all random fields $v:\Omega\times [0,T]\times \mathbf{R}^d\rightarrow \mathbf{R}^d$ such that  $v$ is $\cO_T\otimes\cB(\mathbf{R}^d)$-measurable and $\bP$-a.s.\ $r_{1}^{-\lambda}(\cdot)v_{\cdot}(\cdot)$ is a $C([0,T];\cC^{\beta'}(\bR^d;\bR^d))$ for a real  number $\lambda>0$.

We introduce the following assumption for a real number $\beta >2$.
\begin{assumption}[$\beta$]\label{asm:propflowregwcorrec}
There is a constant $N_0$ such
that for all $(\omega,t)\in \Omega\times [0,T]$,
$$
|r_1^{-1}b_t|_0+|r_1^{-1}\sigma_t|_0+
|\nabla b_t|_{\beta -1}+| \nabla \sigma_t|_{\beta -1}\leq N_{0}.
$$
\end{assumption}

\begin{theorem}\label{thm:SPDEEx}Let Assumption  \ref{asm:propflowregwcorrec}$(\beta)$ hold for some $\beta >2$.  Then for each stopping time $\tau\le T$ and  $\beta'\in [1,\beta)$, there exists a unique process $u(\tau)$ in $\mathfrak{C}_{cts}^{\beta'}(\bR^d;\bR^d)$ that solves  \eqref{eq:SPDEIntro}. Moreover, $\bP$-a.s.\ $u_t(\tau,x)=Y^{-1}_t(\tau,x)$ for all $(t,x)\in [0,T]\times\mathbf{R}^d$ and for each $\epsilon>0$ and $p\geq 2,$ there is  a
constant $N=N(d,p,N_{0},T,\beta',\epsilon)$ such that 
\begin{equation*}
\bE\left[\sup_{s,t\leq T}|r_1^{-(1+\epsilon)}u_t(s)| _{0}^{p}\right]+\bE\left[\sup_{s,t\leq T}|r_1^{-\epsilon}\nabla u_t(s)| _{\beta'-1}^{p}\right]\leq N.
\end{equation*}
\end{theorem}
\begin{remark}\label{rem:sigmazero}
It is  clear by the proof of this theorem that if $\sigma\equiv 0$, then we only need to assume that  Assumption  \ref{asm:propflowregwcorrec} $(\beta)$  holds for some $\beta >1$.
\end{remark}

Now, consider the  SPDE  given by
\begin{align}\label{eq:SPDE}
d\bar u_t (x)&=\left(\frac{1}{2}\sigma
^{i\varrho}_t(x) \sigma^{j\varrho}_t(x)\partial_{ij}\bar u_t(x)+ b^i_t(x)\partial_i\bar u_t(x)\right)d t +\sigma^{i\varrho}_t(x) \partial_i\bar u_{t}(x)dw^{\varrho}_t,\;\;\tau <t\le T,\notag \\
\bar u_t(x) &=x,\;\;t\leq \tau.
\end{align}

This SPDE differs from the one given in \eqref{eq:SPDEIntro} by the first-order coefficient in the drift. In order to obtain an existence and uniqueness theorem for this equation, we have to impose  additional assumptions on $\sigma$.

We introduce the following assumption for a real number $\beta >2$.

\begin{assumption}[$\beta$]\label{asm:propflowregwocorrec}
There is a constant $N_0>0$ such
that for all $(\omega,t)\in \Omega\times [0,T]$,
$$
|r_1^{-1}b_t|_0+
|\nabla b_t|_{\beta -1}+|  \sigma_t|_{\beta+1 }\leq N_{0}.
$$ 
\end{assumption}
For each $\tau\le T$, consider the stochastic flow  $\hat Y_t=\hat Y_t(\tau,x)$,
 $(t,x) \in [0,T]\times \mathbf{R}^d$, generated by the SDE
 \begin{align} 
 d\bar Y_{t} &=-\hat b_t(\bar Y_t)dt-\sigma^{\varrho}_t (\bar Y_{t}) dw^{\varrho}_{t}, \;\;\tau <t\le T,\\
 Y_{t} &=x,\;\;t\leq \tau.\notag
 \end{align}
If Assumption \ref{asm:propflowregwocorrec}($\beta$) holds for some $\beta>2$, then for all $(\omega,t,x)\in \Omega\times [0,T]\times \bR^d$,
 $$
 |\hat b_t(x)|\le |b_t(x)|+|\sigma_t(x)|\nabla \sigma_t(x)|\le N_0(N_0+1)+N_0|x|
 $$
 and
 $$
 |\nabla \hat b_t|_{\beta-1}\le |\nabla  b_t|_{\beta-1}+| \sigma_t|_{\beta-1}   |\nabla^2 \sigma_t|_{\beta-1}+|\nabla \sigma_t|_{\beta-1}^2\le N_0+2N_0^2,
 $$
 which immediately implies the following corollary of Theorem \ref{thm:SPDEEx}.

\begin{corollary}\label{cor:SPDEExredform}Let Assumption \ref{asm:propflowregwocorrec}$(\beta)$  hold for some $\beta >2$. Then for each stopping time $\tau\le T$ and  $\beta'\in [1,\beta)$, there exists a unique process $\bar u(\tau)$ in $\mathfrak{C}_{cts}^{\beta'}(\bR^d;\bR^d)$ that solves  \eqref{eq:SPDE}. Moreover, $\bP$-a.s.\ $\bar u_t(\tau,x)=\bar Y^{-1}_t(\tau,x)$ for all $(t,x)\in [0,T]\times\mathbf{R}^d$ and for each $\epsilon>0$ and $p\geq 2,$ there is  a
constant $N=N(d,p,N_{0},T,\beta',\epsilon)$ such that 
\begin{equation*}
\bE\left[\sup_{s,t\leq T}|r_1^{-(1+\epsilon)}\bar u_t(s)| _{0}^{p}\right]+\bE\left[\sup_{s,t\leq T}|r_1^{-\epsilon}\nabla \bar u_t(s)| _{\beta'-1}^{p}\right]\leq N.
\end{equation*}
\end{corollary}

\section{Properties of stochastic flows}

\subsection{Homeomorphism property of  flows}

In this subsection, we collect some results about flows of jump SDEs that we will need. In particular, we present sufficient conditions that guarantee the homeomorphism property  of flows of jump SDEs. First, let us introduce the following assumption, which is the usual linear growth and Lipschitz  condition on the coefficients $b,\sigma$, and $H$ of the SDE \eqref{eq:SDEIntro}.

\begin{assumption}
\label{asm:lineargrowthlipschitz}There is a constant $N_{0}>0$ such that for all $(\omega ,t,x,y)\in \Omega\times[0,T]\times \bR^{2d}$,
\begin{align*}
|b_t(x)|+|\sigma_t (x)| &\leq N_{0}(1+|x|), \\
|b_t(x)-b_t(y)|+|\sigma_t (x)-\sigma_t(y)| &\leq
N_{0}|x-y|.
\end{align*}%
Moreover, for all $(\omega ,t,x,y,z)\in \Omega\times[0,T]\times  \bR^{2d}\times Z,$
\begin{align*}
|H_t(x,z)|&\leq
K_{1}(t,z)(1+|x|),\\
|H_t(x,z)-H_t(y,z)| &\leq K_{2}(t,z)|
x-y|, 
\end{align*}%
where $K_1,K_2: \Omega \times[ 0,T]\times Z\rightarrow \mathbf{R}_+$ \textit{are}  $\mathcal{P}_{T}\otimes \mathcal{Z}$-measurable functions satisfying 
\begin{equation*}
K_{1}(t,z)+K_2(t,z)+\int_{Z}\left(K_{1}(t,z)^{2}+K_2(t,z)^2\right)\pi (dz)\leq
N_{0},
\end{equation*}
for all $(\omega,t,z)\in \Omega\times [0,T]\times Z$.
\end{assumption} 

It is well-known that under this assumption  that there exists a unique strong solution $X_t(s,x)$ of \eqref{eq:SDEIntro} (see e.g. Theorem 3.1 in \cite{Ku04}). We will also make use of the following assumption.

\begin{assumption}\label{asm:Hdiffeoasm}
For all $(\omega,t,x,z)\in \Omega\times [0,T]\times \bR^d\times Z$, $H_t(x,z) $ is differentiable in $x$, and there are constants $\eta\in (0,1)$  and $N_{\kappa}>0$  such that for all $(\omega ,t,x, z)\in$ $ \{(\omega ,t,x,z)\in \Omega \times
[ 0,T]\times \mathbf{R}^d\times Z$ $:|\nabla H_{t}$ $(\omega ,x,z)|>\eta \},$  
\begin{equation*}
\left\vert\left( I_{d}+\nabla
H_t(x,z)\right) ^{-1}\right\vert \leq N_{\kappa}.
\end{equation*}
\end{assumption}
The coming lemma shows that under  Assumptions \ref{asm:lineargrowthlipschitz} and \ref{asm:Hdiffeoasm},  the mapping $x+H_t(x,z)$ from $\mathbf{R}^d$ to $\mathbf{R}^d$ is a diffeomorphism and the gradient of inverse map is bounded. 

\begin{lemma}
\label{lem:Hprop}  Let  Assumptions \ref{asm:lineargrowthlipschitz} and \ref{asm:Hdiffeoasm} hold. For each $(\omega,t,z)\in \Omega\times[0,T]\times Z$, the mapping $\tilde{H}_t(\cdot,z):\mathbf{R}^d\rightarrow\mathbf{R}^d$ defined by $\tilde{H}_t(x,z):=x+H_t(x,z)$ is a 
diffeomorphism and
$$
| \tilde{H}_t^{-1}(x,z)|\leq  \bar N N_0+\bar N|x|\quad \textrm{and} \quad |\nabla \tilde{H}_t^{-1}(x,z)| \leq \bar N,
$$
where $\bar N:=(1-\eta)^{-1}\vee N_0.$
\end{lemma}
\begin{proof}
(1)  On the set $(\omega ,t,x,z)\in \{(\omega ,t,x,z)\in \Omega \times
[ 0,T]\times\mathbf{R}^d\times  Z:|\nabla H_t(\omega,x,z)|\le \eta \}$, we have
\begin{equation*}
|\kappa_t (\omega,x,z)| \leq \left| I_d+\sum_{n=1}^{\infty }(-1)^{n}[\nabla H_t(\omega,x,z)]^{n}\right|\le \frac{1}{1-\eta}.
\end{equation*}%
It follows from  Assumption \ref{asm:Hdiffeoasm} that  for all $\omega,t,x,$ and $z$, the mapping $\nabla \tilde{H}_t(x,z)$ has a  bounded inverse. Therefore, by Theorem 0.2 in \cite{DeHoIm13}   the mapping $\tilde{H}_t(\cdot,z):\mathbf{R}^{d}\rightarrow \mathbf{R}^{d}$ is a
global diffeomorphism. Moreover,  for all  $\omega,t,x$ and $z$,
$$
| \tilde{H}_t^{-1}(x,z)-\tilde{H}_t^{-1}(y,z)|\le \bar N|x-y|,
$$
which yields
$$
|\tilde{H}_t(x,z)-\tilde{H}_t(y,z)|\ge \bar N^{-1}|x-y| \quad \Longrightarrow \quad  |\tilde{H}_t(x,z)|+K_1(t,z)\ge \bar N^{-1}|x|,
 $$
and hence
$$
 |\tilde{H}_t^{-1}(x,z)|\le \bar NK_1(t,z)+\bar N|x|\le \bar NN_0+\bar N|x|.
 $$
\end{proof}

The following estimates are essential in the proof of the homeomorphic property of the flow and the derivation of moment estimates of the inverse flow.  We refer the reader to  Theorem 3.2 and  Lemmas 3.7 and 3.9 in \cite{Ku04} and Lemma 4.5.6 in \cite{Ku97} ($H\equiv 0$ case)  for the proof of the following lemma.

\begin{lemma}
\label{lem:Direct Flow Estimates}Let  Assumption \ref{asm:lineargrowthlipschitz} hold. 
\begin{tightenumerate}
\item For each $p\geq 2,$ there is  a constant $N=N(p,N_{0},T)$ such that
for all $s,\bar s\in [0,T]$ and $x,y\in \mathbf{R}^d,$ 
\begin{equation}\label{ineq:growthdirectposp}
\bE\left[\sup_{t\leq T}r_{1}(X_{t}(s,x)^{p})\right]\leq Nr_{1}(x)^{p},
\end{equation}
\begin{equation}\label{ineq:estdirectdifftposp}
\bE\left[\sup_{t\leq T}|X_{t}(s,x) -X_{t}(s,y)
|^{p} \right]\leq N|x-y|^{p}.
\end{equation}
\item If  Assumption \ref{asm:Hdiffeoasm} holds, then for each $p\in\mathbf{R}$, there is a
constant $N=N(p,N_{0},T,\eta,N_{\kappa})$ such that for all $s\in[0,T]$ and $%
x,y\in\mathbf{R}^d$, 
\begin{equation}\label{ineq:estdirectgrowthnegp}
\bE\left[\sup_{t\leq T}r_{1}(X_{t}(s,x)^{p}\right]\leq Nr_{1}(x)^{p},
\end{equation}
and 
\begin{equation}\label{ineq:estdirectdiffnegp}
\bE\left[\sup_{t\leq T}|X_{t}(s,x) -X_{t}(s,Y)
|^{p}\right]\leq N|x-y|^{p}.  
\end{equation}
\end{tightenumerate}
\end{lemma}

In the next proposition, we collect some facts  about the homeomorphic property of the flow.  Actually, the homeomorphism property  has been shown in \cite{QiZh08} to hold under the log-Lipschitz condition (i.e. one uses Bihari's inequality instead of Gronwall's inequality), but we do not pursue this here.

\begin{proposition}
\label{prop:homeomorphism}Let  Assumptions \ref{asm:lineargrowthlipschitz} and \ref{asm:Hdiffeoasm} hold.
\begin{tightenumerate}
\item There exists a modification of the strong solution $X_t(s,x),$ $(s,t,x)\in [0,T]^2\times\mathbf{R}^d$, of \eqref{eq:SDEIntro}, also denoted by $X_t(s,x)$, that is c\`{a}dl%
\`{a}g  in $s$ and $t$ and continuous in $x$. 
Moreover, for each stopping time $\tau\le T$,  $\bP$-a.s.\ for all $t\in
[0, T]$, the mappings $X_{t}(\tau ,\cdot),X_{t-}(\tau,\cdot):\mathbf{R}^d\rightarrow\mathbf{R}^d$ are homeomorphisms and the inverse of $%
X_{t}(\tau,\cdot ),$ denoted by $X_t^{-1}(\tau,\cdot )$, is c\`{a}dl%
\`{a}g in $t$ and continuous in $x$, and $X_{t-}^{-1}(\tau,\cdot )$ coincides with
the inverse of $X_{t-}(\tau,\cdot ).$  In particular,  if $(x_{n})_{n\ge 1}$ is a
sequence in $\mathbf{R}^d$ such that $\lim_{n\rightarrow \infty }x_{n}=x$
for some $x\in\mathbf{R}^d$, then $\bP$-a.s.\, 
\begin{equation*}
\lim_{n\rightarrow \infty }\sup_{t\leq T}|X_{t}^{-1}(\tau
,x_{n})-X_{t}^{-1}(\tau,x)|=0.
\end{equation*}
Furthermore,  for each $\beta'\in [0,1)$, $\bP$-a.s.\ $X(\tau,\cdot) \in D([0,T];\cC_{loc}^{\beta'}(\bR^d;\bR^d))$ and  for all  $\epsilon>0$ and  $p\ge 2$, there is a constant $N=N(d,p,N_0,T,\beta',\epsilon)$ such that 
\begin{equation}\label{ineq:estimateofdirect}
\bE\left[\sup_{t\le T} |r_1^{-(1+\epsilon)}X_t(\tau)|_{\beta'}^p\right]\le N.
\end{equation}
\item If $H\equiv 0$, then $\bP$-a.s.\ for all $s,t\in[0,T]$, the $X_{t}(s,x)$ and  $X_{t}^{-1}(s,x)$ are continuous in $s,t,$ and $x.$ Moreover, for each $\beta'\in [0,1]$, $\bP$-a.s.\ $X(\cdot,\cdot) \in C([0,T]^2;\cC_{loc}^{\beta'}(\bR^d;\bR^d))$ and  for each  $\epsilon>0$ and  $p\ge 2$, there is a constant $N=N(d,p,N_0,T,\beta',\epsilon)$ such that 
\begin{equation}\label{ineq:estimateofdirectcts}
\bE\left[\sup_{s,t\le T} |r_1^{-(1+\epsilon)}X_t(s)|_{\beta'}^p\right]\le N.
\end{equation}
\end{tightenumerate}
\end{proposition}
\begin{proof}
(1) Owing to Assumptions  \ref{asm:lineargrowthlipschitz} and \ref{asm:Hdiffeoasm}, by Lemma \ref{lem:Hprop},  for all $\omega,t$ and $z$, the process $\tilde{H}_t(x,z):=x+H_t(x,z)$  is a homeomorphism  (in fact, it is a diffeomorphism) in $x$ and $\tilde{H}^{-1}_t(x,z)$ has linear growth and is Lipschitz.
This implies that assumptions of Theorem 3.5 in \cite{Ku04} hold and hence there is modification of $X_t(s,x)$, denoted $X_t(s,x)$, such that   for all $s\in [0,T]$, $\bP$-a.s.\ for all $t\in [0,T]$, $X_t(s,\cdot )$ is a homeomorphism. Following \cite{Ku04}, for each $(s,t,x)\in [0,T]^2\times\mathbf{R}^d$, we set
\begin{equation}\label{def:flowdef}
\bar X_t(s,x) =\left\{ \begin{array}{cc}
x & t\le s\\
X_t(0,X_s^{-1}(0,x)) & t\ge s,
\end{array}\right.
\end{equation}
and remark that $\bP$-a.s.\ $\bar X_t(s,x)$ is c\`{a}dl\`{a}g  in $s$ and $t$ and continuous in $x$, and $\bP$-a.s.\   for all $(s,t)\in [0,T]^2$, $\bar X_t(s,\cdot )$ is a homeomorphism, and $\bar X_t(s,x)$ is a version of $X_t(s,x)$ (the equation started at $s$).  Fix a stopping time $\tau\le T$. We will now show that $\bar X_t(\tau,x)=\bar X_t(s,x)|_{s=\tau}$ (i.e. $\bar X_t(s,x)$ evaluated at $s=\tau$) is a version of $X_t(\tau,x)$. Define the sequence of stopping times $(\tau _{n})_{n\ge 1}$
by 
$$\tau _{n}=\sum_{k=1}^{n-1}\frac{kT}{n}\mathbf{1}_{\left\{ \frac{(k-1)T}{n}\le \tau < 
\frac{kT}{n}\right\} }+T\mathbf{1}_{\left\{\tau\ge \frac{(n-1)T}{n}\right\}}.$$  For each $n$ and $x$, let $X_t^{(n)}=X_{t}^{(n)}(x)=\bar X_t(\tau_n,x)$, $t\in [0,T]$. It follows that  for each $n$, $t$, and $x$, $\bP$-a.s.\ for all $k\in \{1,\ldots,n\}$, $$X_t^{(n)}(x)\mathbf{1}_{\{\tau_n=\frac{kT}{n}\}}= X_t\left(\frac{kT}{n},x\right)\mathbf{1}_{\{\tau_n=\frac{kT}{n}\}} ,$$ and hence  
\begin{align*}
X^{(n)}_t(x)\mathbf{1}_{\{\tau_n=\frac{kT}{n}\}}&=\mathbf{1}_{\{\tau_n=\frac{kT}{n}\}}x+\mathbf{1}_{\{\tau_n=\frac{kT}{n}\}}\int_{]\frac{kT}{n},\frac{kT}{n}\vee t]}b_r(X_{r}^{(n)}(x))dr\\
&\quad +\mathbf{1}_{\{\tau_n=\frac{kT}{n}\}}\int_{]\frac{kT}{n},\frac{kT}{n}\vee t]}\sigma^{\varrho}
_r(X_{r}^{(n)}(x))dw^{\varrho}_{r}\\
&+\mathbf{1}_{\{\tau_n=\frac{kT}{n}\}}\int_{]\frac{kT}{n},\frac{kT}{n}\vee t]}\int_Z H
_r(X_{r}^{(n)}(x),z)q(dr,dz).
\end{align*}
Since $\Omega$ is the disjoint union of the sets $\left\{ \tau _{n}=%
\frac{kT}{n}\right\}$, $k\in\{1,\ldots,n\}$,  it follows that $X_{t}^{(n)}(x)$ solves
\begin{align*}
X_{t}^{(n)}(x)&=x+\int_{]\tau _{n},\tau _{n}\vee t]}b_r
(X_{r}^{(n)}(x))dr+\int_{]\tau _{n},\tau _{n}\vee t]}\sigma^{\varrho}_r
(X_{r}^{(n)}(x))dw ^{\varrho}_{r}\\
&\quad +\int_{]\tau _{n},\tau _{n}\vee t]}\int_ZH_r
(X_{r}^{(n)}(x),z)q(dr,dz).
\end{align*}%
Thus, by uniqueness, we have that  for each $t$ and $x$,   $\bP$-a.s.\ $\bar X_{t}(\tau_n,x)=X^{(n)}_t(x)=X_{t}(\tau
_{n},x)$.  It is easy to check that for each $t$ and $x$, $\bP$-a.s.\, $X_t(\tau_n,x)$ converges to $X_t(\tau,x)$ as $n$ tends to infinity. Since $\bar X_t(s,x)$ is c\`{a}dl\`{a}g  in $s$, we have that $\bar X_t(\tau_n,x)$ converges to $\bar X_t(\tau,x)$ as $n$ tends to infinity. Therefore, $\bar X_t(\tau,x)$ is a version of $X_{t}(\tau,x)$ for all $t$ and $x$.  We identify $X_t(s,x)$ and $\bar X_t(s,x)$ for all $(s,t,x)\in [0,T]^2\times \mathbf{R}^d$.  Using  Lemma \ref{lem:Direct Flow Estimates}(1) and Corollary \ref{cor:Kolmogorov Embedding}, we obtain that $\bP$-a.s  $X_{\cdot}(\tau,\cdot) \in D([0,T];\cC_{loc}^{\beta'}(\bR^d;\bR^d))$ and that the estimate \eqref{ineq:estimateofdirect} holds. Note here that for each $\beta \ge 0$, the  Fr\'echet spaces $D([0,T];\cC_{loc}^{\beta}(\bR^d;\bR^d))$ and $\cC_{loc}^{\beta}(\bR^d;\allowbreak D([0,T];\mathbf{R}^{d}))$  are equivalent.  It follows from the proof of Theorem 3.5 in \cite{Ku04} that for every stopping time  $\bar \tau \le T$, $\bP$-a.s.\
\begin{equation}\label{eq:surjective}
\lim_{|x|\rightarrow\infty }\inf_{t\le T}|X_t(\bar \tau,x)|=\infty.
\end{equation}
Let $(t_{n})\subseteq[0,T]$ and $(x_n)\subseteq\mathbf{R}^d$ be convergent
sequences with limits $t$ and $x$, respectively. First, assume $t_{n}<t$ for
all $n$. By \eqref{eq:surjective}, for every stopping time $\bar \tau\le T$,  $\bP$-a.s.\  the sequence $\left(X_{t_{n}}^{-1}(\bar \tau 
,x_{n})\right) $ is uniformly bounded.  Since $\bP$-a.s.\ $X_{\cdot}(\tau,\cdot)\in D([0,T];\cC^{\beta}(\bR^d;\bR^d))$, $\beta'\in (0,1)$,  we have
\begin{gather*}
\lim_{n\rightarrow\infty} \left(X_{t-}(\bar \tau,X_{t_{n}}^{-1}(\bar \tau 
,x_{n}))-X_{t-}(\bar  \tau  ,X_{t-}^{-1}(\bar \tau  ,x) \right)=\lim_{n\rightarrow\infty} \left(X_{t-}(\bar \tau,X_{t_{n}}^{-1}(\bar \tau 
,x_{n}))-x \right)\\
=\lim_{n\rightarrow\infty} \left(X_{t_n}(\bar \tau,X_{t_{n}}^{-1}(\bar \tau 
,x_{n}))-x \right)=\lim_{n\rightarrow\infty} (x_n-x )=0,
\end{gather*}
which implies
$$\lim_{n\rightarrow\infty }X_{t_{n}}^{-1}(\bar \tau ,x_{n})= X_{t-}^{-1}(\bar \tau,x).$$
 A similar argument is used for $t_{n}>t$. 
(2) It follows from the definition \eqref{def:flowdef} that $\bar X_t(s,x)$ and  $\bar X^{-1}_t(s,x)$ are continuous in $s$, $t$, and $x$.  Moreover, applying  Lemma \ref{lem:Direct Flow Estimates}(1) and Corollary \ref{cor:Kolmogorov Embedding}, we get that $\bP$-a.s.\  $X_{\cdot}(\cdot,\cdot) \in C([0,T]^2;\cC_{loc}^{\beta'}(\bR^d;\bR^d))$ and that the estimate \eqref{ineq:estimateofdirectcts} holds. The continuity of $X_s(\tau,x)$ with respect to $s$ actually plays an important role in the proof  of Theorem \ref{thm:SPDEEx}.
\end{proof}

\subsection{Moment estimates of  inverse flows: Proof of Theorem \ref{thm:diffeoandmomest}}

In this subsection, under Assumption \ref{asm:regularitypropflow} ($\beta$), $\beta\ge 1$,  we derive moment estimates    
for the flow $X_t(\tau,x)$ and its inverse $X_t^{-1}(\tau,x)$  in weighted  H\"{o}lder norms and complete the proof of Theorem \ref{thm:diffeoandmomest}. In particular, we will apply Corollaries \ref{cor:SobolevFull} and \ref{cor:Kolmogorov Embedding} with the Banach spaces $V=D([0,T];\mathbf{R}^{d})$ and $V=C([0,T]^{2};\mathbf{R}^{d})$.

\begin{proposition}
\label{p:Regularity of direct flow}Let Assumption \ref{asm:regularitypropflow}$(\beta)$ hold
for some $\beta>1$
\begin{tightenumerate}
\item For each stopping time $\tau\le T$ and $\beta '\in [1,\beta)$,   $\bP$-a.s.\  $\nabla X_{\cdot}(\tau,\cdot )\in D([0,T];\cC_{loc}^{\beta'-1}(\bR^d;\bR^d))$ and for each $\epsilon>0$ and $p\ge 2$, there is a constant $N=N(d,p,N_{0},T,\beta',\epsilon)$  such that 
\begin{equation}  \label{ineq:momentestdirectjmp}
\bE\left[\sup_{t\leq T}| r_1^{-\epsilon}\nabla X_{t}(\tau)| _{\beta'-1}^{p}\right]\le N.
\end{equation}%
Moreover, for each $p\ge 2$, there is a constant $N=N(d,p,N_{0},\beta,T)$  such that for all multi-indices $%
\gamma$ with $1\le |\gamma| \le \left[ \beta \right] $ and all $x\in 
\mathbf{R}^{d}$, 
\begin{equation} \label{ineq:GradientMomentbd} 
\bE\left[\sup_{t\leq T}|\partial ^{\gamma }X_{t}(\tau
,x)|^{p}\right] \leq N
\end{equation}
and for all multi-indices $\gamma$ with $|\gamma|=[ \beta]^- $
and all $x,y\in \mathbf{R}^{d}$, 
\begin{equation}\label{ineq:GradientMomentdiff}  
\bE\left[\sup_{t\leq T}|\partial ^{\gamma }X_{t}(\tau ,x)
-\partial ^{\gamma }X_{t}( \tau ,y) |^{p}\right] \leq
N|x-y|^{\{\beta\}^+p}.  
\end{equation}

\item If $H\equiv 0$, then  for each $\beta '\in [1,\beta)$,  $\bP$-a.s.\  $\nabla X_{\cdot}(\cdot,\cdot )\in C([0,T]^2;\cC_{loc}^{\beta'-1}(\bR^d;\bR^d))$ and for each $\epsilon>0$ and $p\ge 2$, there is a constant $N=N(d,p,N_{0},T,\beta',\epsilon)$  such that 
\begin{equation}  \label{ineq:momentestdirectcts}
\bE\left[\sup_{s,t\leq T}| r_1^{-(1+\epsilon)}\nabla X_{t}(s)| _{\beta'-1}^{p}\right]\le N.
\end{equation}
Moreover, for each $p\ge 2$, there is a constant $N=N(d,p,N_0,T,\beta) $ such that
for all multi-indices $\gamma$ with $|\gamma|=[\beta]^{-}$ and all $s,\bar s%
\in[0,T]$ and $x\in \mathbf{R}^d$, 
\begin{equation}\label{ineq:derivsmsbar}
\bE\left[\sup_{t\le T}|\partial^{\gamma}X_t(s,x)-\partial^{\gamma}X_t(%
\overline{s},x)|^p\right]\le N|s-\overline{s}|^{p/2}.
\end{equation}
\end{tightenumerate}
\end{proposition}
\begin{proof}
(1) Fix a stopping time  $\tau\le T$ and write $X_t(\tau,x)=X_t(x)$.  First, let us assume that $[\beta]^{-}=1$. It  follows from  Theorem 3.4 in  \cite{Ku04} that  $\bP$-a.s.\ for all $t$, $X_t(\tau,\cdot)$ is continuously differentiable and 
$U_t=\nabla X_{t}(\tau,x) $
satisfies 
\begin{align}\label{eq:GradientofFlow}
dU_t &=\nabla 
b_t(X_{t})U_{t}dt+\nabla \sigma^{\varrho}_t (X_{t-})U_{t}dw^{\varrho}_{t}+
\int_{Z}\nabla H_t(X_{t-},z )U_{t-}q(dt,dz),\;\;\tau <t\le T,\notag\\
\nabla X_t&=I_d,\;\;t\le \tau,
\end{align}
where $I_d$ is the $d\times d$-dimensional identity matrix.  Taking $\lambda =0$ in the estimates (3.10) and (3.11) in Theorem 3.3 in \cite{Ku04}, we obtain  \eqref{ineq:GradientMomentbd} and \eqref{ineq:GradientMomentdiff}. Then applying Corollary \ref{cor:Kolmogorov Embedding} with $V=D([0,T];\mathbf{R}^d)$, we have that  $X_{\cdot}(\cdot)\in D([0,T];\cC^{\beta'}_{loc}(\bR^d;\bR^d))$ fand that  the  \eqref{ineq:momentestdirectjmp} holds.  The  proof for $[\beta]^{-}>1$ follows by induction (see, e.g. the proof of Theorem 6.4 in \cite{Ku97}).

(2)  The estimate \eqref{ineq:derivsmsbar} is given in Theorem 4.6.4 in \cite{Ku97}  in equation (19). 
The remaining items of part (2) then follow in exactly the same way as part (1) with the only exception being that we apply Corollary \ref{cor:Kolmogorov Embedding} with $V=C([0,T]^2;\mathbf{R}^d)$.
\end{proof}
\begin{lemma}\label{lem:gradientinverseest}
Let Assumption \ref{asm:regularitypropflow}$ (\beta)$ hold
for some $\beta>1$.
\begin{tightenumerate}
\item For each stopping time $\tau\le T$ and  $\beta '\in [1,\beta)$,   $\bP$-a.s.\ $\nabla X_{\cdot }(\tau ,\cdot)^{-1}\in 
D([0,T];\allowbreak\cC_{loc}^{\beta'-1}\allowbreak(\bR^d;\bR^d))$ 
and for each $p\geq 2$, there is a constant $N=N(d,p,N_0,T,\eta,N_{\kappa}) $
such that for all $x,y\in \mathbf{R}^{d}$ 
\begin{equation}\label{ineq:momentestgradinvbd}
\bE\left[\sup_{t\leq T}|\nabla X_{t}(\tau ,x)^{-1}|^{p}\right]\leq N
 \end{equation}%
and 
\begin{equation}\label{ineq:momentestgradinvdiff}
\bE\left[\sup_{t\leq T}|\nabla X_{t}(\tau ,x)^{-1}-\nabla X_{t}(\tau
,y)^{-1}|^{p}\right]\leq N|x-y|^{((\beta-1)\wedge 1 )p}.
\end{equation}
\item If $H\equiv0$, then for each $\beta '\in [1,\beta)$,  $\bP$-a.s.\ $\nabla X_{\cdot}(\cdot ,\cdot)^{-1}\in C([0,T]^{2};\cC_{loc}^{\beta'-1}(\bR^d;\bR^d))$ and  for each $p\geq 2$, there is a constant $%
N=N(d,p,N_0,T) $ such that for all $s,\bar s\in [0,T]$ and $x\in \mathbf{R}^d$,
\begin{equation} \label{ineq:momentestgraddiffctsinv}
\bE\left[\sup_{t\le T}|\nabla X_t(s,x)^{-1}-\nabla X_t(s,x)^{-1}|^{p}\right]\leq N|s-%
\bar{s}|^{p/2}. 
\end{equation}
\end{tightenumerate}
\end{lemma}
\begin{proof}
(1) Let $\tau\le T$ be a fixed stopping time and write $X_t(\tau,x)=X_t(x)$.  Using It\^{o}'s formula (see also Lemma 3.12 in \cite{Ku04}), we deduce that $\bar U_t=[\nabla X_t(x)]^{-1}$ satisfies
\begin{align} \label{eq:Inversegradient}
d\bar U_t&=
\bar U_t\left(\nabla \sigma^{\varrho}_t(X_{t-}) \nabla
\sigma^{\varrho}_t(X_{t-}(\tau))\ -\nabla b_t(X_{t})\right)dt -\bar{U}%
_{t}\nabla \sigma^{\varrho}_t(X_{t})dw^{\varrho}_{t} \notag\\
&\quad -\int_{Z}\bar{U}%
_{t-}\nabla H_t(X_{t-},z )(I_d+\nabla H_t(X_{t-},z
))^{-1}q(dt,dz) \notag \\
&\quad +\int_Z\bar{U}_t\nabla H_t(X_{t-},z )^{2}(I_d+\nabla
H_t(X_{t-},z ))^{-1}\pi (dz) dt,\;\;\tau <t\le T,\notag\\
\bar{U}_t&=I_d, \;\; t\le \tau.
\end{align}
Since matrix inversion is a smooth mapping, the coefficients of the linear equation   \eqref{eq:Inversegradient} satisfy the same assumptions as the coefficients of the linear equation \eqref{eq:GradientofFlow}, and hence the derivation of the estimates \eqref{ineq:momentestgradinvbd} and \eqref{ineq:momentestgradinvdiff} proceed in the same way as the analogous estimates for \eqref{eq:GradientofFlow}. To see that  $\bP$-a.s.\ $X_{\cdot}(\cdot)^{-1}\in D([0,T];\cC_{loc}^{\beta
'-1}(\bR^d;\bR^d))$, we only need to note that  $\bP$-a.s.\ $X_{\cdot}(\cdot )\in D([0,T];\cC_{loc}^{\beta
'-1}(\bR^d;\bR^d))$ and  that matrix inversion is  a smooth mapping.  Part (2) follows with the obvious changes. 
\end{proof}

As an immediate corollary, we obtain the diffeomorphism property of the flow 
$X_t(\tau,x)$ under the assumptions  Assumption \ref{asm:regularitypropflow}$(\beta)$, $\beta>1$. 

\begin{corollary}\label{c:DiffeomorphismProperty} Let Assumption \ref{asm:regularitypropflow}$(\beta)$  hold. 
\begin{tightenumerate}
\item For each stopping time $\tau\le T$ and $\beta'\in [1,\beta)$  the mapping $X_{t}(\tau,\cdot):\mathbf{R}%
^{d}\rightarrow \mathbf{R}^{d}$ is a $\cC_{loc}^{\beta'}(\bR^d;\bR^d)$%
-diffeomorphism, $\bP$-a.s.\  $X_{\cdot}(\tau,\cdot),
X_{\cdot}^{-1}(\tau,\cdot )\in D([0,T];\cC_{loc}^{\beta'}(\bR^d;\bR^d))$ and for each $t\in[0,T]$, $X_{t-}^{-1}(\tau)$ coincides
with the inverse of $X_{t-}(\tau)$.

\item If $H\equiv 0$, then for each $\beta'\in [1,\beta)$, $\bP$-a.s.\ $X_{\cdot}(\cdot,\cdot),X^{-1}_{\cdot}(\cdot,\cdot) \in C([0,T]^{2},\cC_{loc}^{\beta'}(\bR^d;\bR^d))$.
\end{tightenumerate}
\end{corollary}

\begin{proof}
(1) Fix a stopping time  $\tau\le T$ and write $X_t(\tau,x)=X_t(x)$.  It follows from  Propositions \ref{prop:homeomorphism} and \ref{p:Regularity of direct flow} that $\bP$-a.s.\ for all $t$, the mappings $X_{t}(\cdot),X_{t-}(\cdot):\mathbf{R}^d\rightarrow \mathbf{R}^d$ are homeomorphisms and $X_{\cdot }(\cdot ) \in D([0,T];\cC^{\beta'}_{loc}(\bR^d;\bR^d))$. Moreover, by Lemma \ref{lem:gradientinverseest}, $\bP$%
-a.s. for all $t$ and $x$, the matrix $\nabla X_{t}\left( \tau ,x\right) $
has an inverse.  Therefore, by Hadamard's Theorem (see, e.g., Theorem 0.2 in \cite{DeHoIm13}), $\bP$-a.s.\ for all $t$, $X_{t}(\cdot)\ $ is a diffeomorphism. Using the chain rule, $\bP$-a.s.\ for all $t$ and $x$,
\begin{equation}  \label{eq:inversematrixidentity}
\nabla X_{t}^{-1}(x)=\nabla X_{t}(X_{t}^{-1}(x))^{-1}.
\end{equation}
Since, by Lemma \ref{lem:gradientinverseest}, $\bP$-a.s.\ $[\nabla X_{\cdot }(\cdot )]^{-1}\in D([0,T];\cC_{loc}^{\beta'-1}(\bR^d;\bR^d))$ and we know that $\bP$-a.s.\ for all $t$,   $X^{-1}_{t}(\cdot)$ is differentiable,  it follows from  \eqref{eq:inversematrixidentity} that  $\bP$-a.s.\ $$\nabla X_{\cdot }(X_{\cdot}^{-1}(\cdot))^{-1}\in D([0,T];\cC_{loc}^{(\beta'-1)\wedge 1}(\bR^d;\bR^d)).$$ One then proceeds inductively to complete the proof.  Making the obvious changes in
the proof of part (1), we obtain part (2).
 \end{proof}
We conclude with a derivation of H\"{o}lder moment estimates of the inverse
flow $X_t^{-1}(\tau,x)$, which will complete the proof of Theorem \ref{thm:diffeoandmomest}.
\begin{proof}[Proof of Theorem \ref{thm:diffeoandmomest}]
(1)  Fix a stopping time  $\tau\le T$ and write $X_t(\tau,x)=X_t(x)$.  Fix $\epsilon>0$. First, let us assume that $\left[ \beta \right] ^{-}=1$.  Set  $J_{t}(x)=|\det \nabla
X_{t}(x)|$.  It is clear from \eqref{ineq:GradientMomentbd}  that for each $p\ge 2$ and $x$, there is a constant $N=N(d,p,N_{0},T)$   such that
\begin{equation}\label{ineq:estofdeterm}
\bE[\sup_{t\le T}|J_t(x)|^p]\le N.
\end{equation}
Using the change of variable $(\bar{x},\bar y)=(X^{-1}_{t}(x),X^{-1}_t(y))$,  Fatou's lemma, Fubini's theorem, H\"{o}lder's
inequality, and the inequalities \eqref{ineq:estdirectgrowthnegp}, \eqref{ineq:estofdeterm},  \eqref{ineq:weightest}, \eqref{ineq:estdirectdifftposp}, and \eqref{ineq:estdirectdiffnegp},  for any $\delta\in (0,1]$ and $p>\frac{d}{\epsilon}$, we obtain  that there is a constant  $N=N(d,p,N_{0},T,\delta,\eta,N_{\kappa},\epsilon)$ such that 
\begin{align*}
\bE\sup_{t\leq T}\int_{\mathbf{R}^{d}}|r_{1}(x)^{-(1+\epsilon)}X_{t}^{-1}(x)|^pdx &\leq \int_{\mathbf{R}
^{d}}|\bar{x}| ^{p}\bE\sup_{t\leq
T}[r_{1}(X_{t}( \bar{x}) )^{-p(1+\epsilon)}J_{t}(\bar{x})]d\bar{x}\\
&\leq N\bE\int_{\mathbf{R}^{d}}r_1(\bar{x})
^{-p\epsilon}d\bar{x}\le N
\end{align*}
and
\begin{gather*}
\bE\sup_{t\leq T} \int_{|x-y|<1}\frac{%
|r_1^{-(1+\epsilon)}(x) X_{t}^{-1}(x)-r_1^{-(1+\epsilon)}(y) X_{t}^{-1}(y)|^{p}
}{|x-y|^{2d+\delta p}}dxdy\\
\le  \int_{|\bar x-\bar y|<1}\bE\sup_{t\leq T}
 \left[\frac{r_1^{-p(1+\epsilon)}(X_{t}(\bar{x}))|\bar x-\bar y|^pJ_{t}(\bar{x})J_t(\bar{y})}{|X_{t}(\bar{x})-X_{t}( \bar{y})|^{2d+\delta p}}\right] d\bar{x}d\bar{y}\\
+ \int_{|\bar x-\bar y|<1}\bE\sup_{t\leq T}
  \left[\frac{|\bar y|^p|r_1^{-(1+\epsilon)}(X_{t}(\bar{x}))-r_1^{-(1+\epsilon)}(X_{t}(\bar{y}))|^pJ_{t}(\bar{x})J_t(\bar{y})}{|X_{t}(\bar{x})-X_{t}( \bar{y})|^{2d+\delta p}}\right] d\bar{x}d\bar{y}\\
\le  N\int_{|\bar{x}-\bar{y}|<1}\frac{r_{1}(\bar{x})^{-p(1+\epsilon) }}{|\bar{x}-\bar{y}|^{2d-(1-\delta 
)p}}d\bar{x}d\bar{y}+N\int_{|\bar{x}-\bar{y}|<1}\frac{r_1(\bar x)^{-p(1+\epsilon)}+r_1(\bar y)^{-p(1+\epsilon)}}{|\bar{x}-\bar{y}|^{2d-(1-\delta
)p}}d\bar{x}d\bar{y}\le N.
\end{gather*}
Similarly,  making use of the inequalities  \eqref{ineq:estdirectgrowthnegp}, \eqref{ineq:estofdeterm},  \eqref{ineq:weightest}, \eqref{ineq:estdirectdifftposp}, \eqref{ineq:estdirectdiffnegp}, \eqref{ineq:momentestgradinvbd},  and  \eqref{ineq:momentestgradinvdiff},  for any $p>\frac{d}{\epsilon}\vee\frac{d}{\beta -\beta '}\vee \frac{d}{2 -\beta '}$,  we get
\begin{align*}
\bE\sup_{t\leq T}\int_{\mathbf{R}^{d}}|r^{-\epsilon}(x)\nabla
X_{t}^{-1}(x) |^{p}dx&\le \int_{\mathbf{R}^{d}}\bE\sup_{t\leq T}[r_{1}(X_{t}(
\bar{x}) )^{-p\epsilon}| [\nabla
X_{t}(\bar{x})]^{-1}|^{p}J_{t}(\bar{x})]d\bar{x}\\
&\leq N\bE\int_{\mathbf{R}%
^{d}}r_{1}(\bar{x})^{-p\epsilon}d\bar{x}\le N
\end{align*}
and
\begin{gather*}
\bE\sup_{t\leq T} \int_{|x-y|<1}\frac{%
|r_1^{-\epsilon}(x)\nabla X_{t}^{-1}(x)-r_1^{-\epsilon}(y)\nabla X_{t}^{-1}(y)|^{p}
}{|x-y|^{2d+(\beta'-1)p}}dxdy\\
\leq \int_{|\bar x-\bar y|<1}\bE\sup_{t\leq T}
\left[\frac{|r_1^{-\epsilon}(X_{t}(\bar{x}))[\nabla X_{t}(\bar{x})]^{-1}-r_1^{-\epsilon}(X_{t}(\bar{y}))[\nabla X_{t}(
\bar{y})]^{-1}|^{p}J_{t}(\bar{x})J_t(\bar{y})}{|X_{t}(\bar{x})-X_{t}( \bar{y})|^{2d+(\beta'-1)p}}\right] d\bar{x}d\bar{y}
\end{gather*}
$$
\le  N\int_{|\bar{x}-\bar{y}|<1}\frac{r_{1}(\bar{x})^{-p\epsilon }}{|\bar{x}-\bar{y}|^{2d-(\beta-\beta' 
)p}}d\bar{x}d\bar{y}+N\int_{|\bar{x}-\bar{y}|<1}\frac{r_1(\bar x)^{-p\epsilon}+r_1(\bar y)^{-p \epsilon}}{|\bar{x}-\bar{y}|^{2d-(2-\beta')
p}}d\bar{x}d\bar{y}\le N,
$$
where $N=N(d,p,N_{0},T,\beta',\eta,N_{\kappa},\epsilon)$ is a positive constant.
Therefore, combining the above estimates and applying Corollary \ref{cor:SobolevFull}, we have that for all $p\ge 2$, there is f
a constant $N=N(d,p,N_{0},T,\beta',\eta,N_{\kappa},\epsilon)$, such that
$$
\bE\left[\sup_{t\leq T}|r_{1}^{-(1+\epsilon )}X^{-1}_{t}(\tau )|_{0}^{p}\right]+\bE\left[\sup_{t\leq T}|r_{1}^{-\epsilon }\nabla X^{-1}_{t}(\tau )|_{\beta'-1}^{p}\right]\leq N.
$$
 It is well-known that  the the inverse map $\mathfrak{I}$ on the set of invertible $d\times d$-dimensional matrices is infinitely 
differentiable and for each $n$, there is a constant $N=N(n,d)$ such that for all invertible matrices $M$, the $%
n$th derivative of $\mathfrak{I}$ evaluated at $M$, denoted $\mathfrak{I} ^{(n)}(M)$, satisfies
\begin{equation}\label{ineq:inversemapderivative}
\left\vert \mathfrak{I} ^{(n)}(M)\right\vert \leq N|M^{-n-1}|\le N\left\vert
M^{-1}\right\vert ^{n+1}.  
\end{equation}
We claim that for each $n$ and every multi-index $\gamma$  with $|\gamma|= n$, the components of  $\partial^{\gamma} X_t^{-1}(x)$ are a polynomial in terms of the entries of  $[\nabla X_t(X_t^{-1}(x))]^{-1}$ and $\partial^{\gamma'} \nabla X_t(X_t^{-1}(x))$ for all multi-indices $\gamma'$ with $1\le |\gamma'|\le n-1$. 
Assume that statement holds for some $n$. By the chain rule, for each $\omega,t,$ and $x$, we have
$$
\nabla  (\nabla X_t(X_t^{-1}(x))^{-1})
=\mathfrak{I}^{(1)}(\nabla X _t(X_t^{-1}(x)))\nabla^2 X_t(X_t^{-1}(x))\nabla X_t (X_t^{-1}(x))^{-1}
$$
and for all multi-indices $\gamma$ with $1\le |\gamma'|\le n-1$,  we have
$$
\nabla (\partial^{\gamma'} \nabla X_t(X_t^{-1}(x)))= \partial^{\gamma'} \nabla^2 X_t(X_t^{-1}(x))\nabla X_t(X_t^{-1}(x))^{-1},
$$
where $\nabla ^2 X_t(X_t^{-1}(x))$ is the tensor of second-order derivatives of $X_t(\cdot)$ evaluated at $X_t^{-1}(x)$.
This implies that for every multi-index $\gamma$  with $|\gamma|= n+1$, the components of  $\partial^{\gamma} X_t^{-1}(x)$ are a polynomial in terms of the entries of  $\nabla X_t(X_t^{-1}(x))^{-1}$ and $\partial^{\gamma'} \nabla X_t(X_t^{-1}(x))$ for all multi-indices $\gamma'$ with $1\le |\gamma'|\le n$. 
By  induction, the claim is true. Therefore, for $[\beta]^-\ge 2$, using \eqref{ineq:GradientMomentbd} and \eqref{ineq:GradientMomentdiff}, we  obtain the moment estimates for the inverse flow in the almost exact same way we did for  $[\beta]^-=1$. Making the obvious changes in
the proof of part (1), we obtain part (2). This completes the proof of Theorem \ref{thm:diffeoandmomest}.
\end{proof}

\subsection{Strong limit of a sequence of flows: Proof of Theorem \ref{thm:stronglimit}}
\begin{proof}[Proof of Theorem \protect\ref{thm:stronglimit}]
Let $\tau\le T$ be a fixed stopping time and write $X_t(\tau,x)=X_t(x)$.    For each $n$, let 
\begin{equation*}
Z^{(n)}_t(x) =X^{(n)}_t(x)-X_{t}(x), \;(t,x)\in [0,T]\times\mathbf{R}^d.
\end{equation*}
Throughout the proof we denote by $(\delta _{n})_{n\ge 1}$ a deterministic
sequence with $\delta _{n}\rightarrow 0$ as $n\rightarrow \infty $ that may change from line to line. Let $N=N(p,N_{0},T)$ be a positive constant, which may change from line to line. By virtue of Theorem 2.1 in \cite{Ku04}  and \eqref{ineq:growthdirectposp}, for all $p\geq 2$ and $t,x$ and $n,$ we have
\begin{equation*}
\bE\left[\sup_{s\leq t}|Z^{(n)}_s(x)|^{p}\right]\leq N\bE%
\int_{]0,t]}|Z^{(n)}_s(x)|^{p}ds+ N\delta _{n}r_{1}(x)^{p}.
\end{equation*}
Since the right-hand-side is finite by \eqref{ineq:growthdirectposp}, applying Gronwall's lemma we  get that for all $x$ and $n$, 
\begin{equation} \label{ineq:estZn}
\bE[\sup_{t\leq T}|Z^{(n)}_t(x)|^{p}]\leq N\delta
_{n}r_{1}(x)^{p}. 
\end{equation}
Similarly, by \eqref{ineq:GradientMomentbd},  
 for all $x$ and $n,$ we have
\begin{equation}\label{ineq:gradboundZn}
\bE\left[\sup_{t\leq T}|\nabla Z^{(n)}_t(x)|^{p}\right]\leq N\delta _{n}.
\end{equation}
Using  \eqref{ineq:GradientMomentbd}, for all $x,y,$ and $n$, we obtain
\begin{equation*}
\bE\left[\sup_{t\leq T}|Z^{(n)}_t(x)-Z^{(n)}(y)|^{p}\right]\le|x-y|^p \bE\sup_{t\leq T}\int_{0}^{1}|\nabla
Z^{(n)}_t(y+\theta (x-y))|^pd\theta \leq N|x-y|^{p}.
\end{equation*}%
It follows immediately from \eqref{ineq:GradientMomentdiff} that  for all $x,y,$ and $n$, 
$$
\bE[\sup_{t\le T}]|\nabla Z^{(n)}_t(x)-\nabla Z^{(n)}_t(y)|^{p}]\leq N|x-y|^{(\beta-1)\vee 1}.
$$
Thus, by Corollary \ref{cor:SobEmbLimit}, we have 
\begin{equation}\label{eq:directsequence}
\lim_{n\rightarrow \infty }\left(\bE\left[\sup_{t\leq T}|r_{1}^{-(1+\epsilon
)}X_{t}^{(n)}-r_{1}^{-(1+\epsilon )}X_{t}|_{0}^{p}\right]+\bE%
\left[\sup_{t\leq T}|r_{1}^{-\epsilon }\nabla X_{t}^{(n)}-r_{1}^{-\epsilon
}\nabla X_{t}|_{0}^{p}\right]\right)=0.
\end{equation}
Owing to a standard interpolation inequality for H\"{o}lder spaces (see, e.g. Lemma 6.32 in \cite{GiTr01}), for each $\delta \in (0,1)$ and  $\bar \beta \in (\beta ',\beta)$, there is a
constant $N(\delta)$ such that%
\begin{align*}
\bE\left[\sup_{t\leq T}|r_{1}^{-\epsilon }\nabla X_{t}^{(n)}-r_{1}^{-\epsilon }\nabla X_{t}|_{\beta '-1}^{p}\right] &\le \delta \bE\left[\sup_{t\leq T}|r_{1}^{-\epsilon }\nabla
X_{t}^{(n)}-\nabla X_{t}|_{\bar{\beta}-1}^{p}\right] \\
&\quad +C_{\delta }\bE\left[\sup_{t\leq T}|r_{1}^{-\epsilon }\nabla
X_{t}^{(n)}-\nabla X_{t}|_{0}^{p}\right],
\end{align*}
and hence since
\begin{equation*}
\sup_{n}\bE\left[\sup_{t\leq T}| r_{1}^{\varepsilon }\nabla
X^{(n)}|_{\bar{\beta}-1}^{p}\right]+\bE\left[\sup_{t\leq
T}|r_{1}^{\varepsilon }\nabla X_t |_{%
\bar{\beta}-1}^{p}\right]<\infty,
\end{equation*}
we have
\begin{equation*}
\lim_{n\rightarrow\infty}\bE[\sup_{t\leq T}|r_{1}^{-\epsilon }\nabla X_{t}^{(n)}-r_{1}^{-\epsilon }\nabla X_{t}|_{\beta '-1}^{p}]=0.
\end{equation*}%
By Theorem \ref{thm:diffeoandmomest}, Corollary \ref{cor:SobEmbLimit},  and the interpolation inequality for H\"{o}lder spaces used above, in order to show
$$
\lim_{n\rightarrow \infty }\bE\left[\sup_{t\leq T} | r_1^{-(1+\epsilon)}X^{(n);-1}_t(\tau )-r_1^{-(1+\epsilon)}X^{-1}_t(\tau)|_{0}^{p}\right]=0
$$
$$\lim_{n\rightarrow \infty }\bE\left[\sup_{t\leq T} |
r_1^{-\epsilon}\nabla X^{(n);-1}_t (\tau) -r_1^{-\epsilon}\nabla X^{-1}_{t}(\tau ) |
_{\beta'-1}^{p}\right]=0,
$$
 it suffices to show that for each $x$,
\begin{equation}\label{eq:pointwiseinverselimit}
d\bP-\lim_{n\rightarrow\infty}\sup_{t\leq T}|X_{t}^{(n);-1}(x)-X_{t}^{-1}(x)|=0  
\end{equation}%
and%
\begin{equation}\label{eq:pointwisegradientlimit}
d\bP-\lim_{n\rightarrow\infty}\sup_{t\leq T}|\nabla X_{t}^{(n);-1}(x)-\nabla
X_{t}^{-1}(x)|= 0.
\end{equation}
For each $n$, define  
\begin{equation*}
 \Theta _{t}^{(n)}(x)=r_{1}(X_{t}^{(n)}(x))^{-1}-r_{1}(X_{t}(x))^{-1}, \; (t,x)\in [ 0,T]\times \mathbf{R}^{d}.
\end{equation*}%
For all $\omega ,t,x,$ and $n,$ we have 
$$
|\Theta _{t}^{(n)}(x)| \leq r_{1}(X_{t}^{(n)}(x))^{-1}r_{1}(X_{t}(x))^{-1}|Z_{t}^{(n)}(x)|,
$$
and hence using H\"{o}lder's inequality, \eqref{ineq:estdirectdiffnegp}, and  \eqref{ineq:estZn},  we obtain that for all $p\ge 2$, $x$,  there is a constant $N=N(p,N_{0},T,\eta,N_{\kappa})$ such that for all  $n$, 
$$
\bE[\sup_{t\leq T}|\Theta _{t}^{(n)}(x)|^{p}] \leq Nr_1(x)^{-p}\delta _{n},
$$
where  $N=N(p,N_{0},T,\eta,N_{\kappa})$ is a  constant.
Furthermore, since 
$$
|\nabla \Theta _{t}^{(n)}(x)| \leq r_{1}(X_{t}^{(n)}(x))^{-2}|\nabla X_{t}^{(n)}(x)|+ r_{1}(X_{t}(x))^{-2}|\nabla X_{t}^{(n)}(x)|,
$$
for all $\omega,t,x,$ and $n$, applying \eqref{ineq:estdirectdiffnegp} and  \eqref{ineq:GradientMomentbd}, for all $p\ge 2$, $x$, and $n$,  we get
$$
\bE\left[\sup_{t\leq T}| r_1(x)\Theta _{t}^{(n)}(x)-r_1(y)\Theta _{t}^{(n)}(y)|^{p}\right] \leq N|x-y|^p.
$$
Then owing to Corollary \ref{cor:SobEmbLimit}, for
each $p\geq 2,$ 
\begin{equation}\label{eq:Thetaconvtozero}
\lim_{n\rightarrow \infty }\bE\left[\sup_{t\leq T}|\Theta _{t}^{(n)}|^{p}_0\right]=0.
\end{equation}%
We claim that for each $R>0$, 
\begin{equation}\label{nf10}
d\bP-\lim_{n\rightarrow\infty}E (n,R):=d\bP-\lim_{n\rightarrow\infty}\sup_{t\leq T}|
X_{t}^{(n);-1}-X_{t}^{-1}|_{0;\left\{ \left\vert x\right\vert \leq R\right\} }=0.  
\end{equation}%
Fix $R>0$. It is enough to show that every subsequence of $E
(n)=E (n,R)$ has a sub-\linebreak subsequence converging to $0$, $\bP$%
-a.s..  Owing to \eqref{eq:directsequence} and \eqref{eq:Thetaconvtozero}, for a given subsequence $(E(n_{k}))$, we can always find sub-subsequence
(still denoted $(E(n_{k}))$ to avoid double indices) such that $\bP$%
-a.s., 
\begin{equation}\label{eq:limitdirectRbar}
\lim_{k\rightarrow \infty }\sup_{t\leq T}|X_{t}^{(n_{k})}-X_{t}|_{\beta';\left\{ \left\vert x\right\vert \leq \bar R\right\} } =0, \;\;\forall \;\bar R>0,
\end{equation}
and
\begin{equation}\label{eq:oneoverlimit}
\lim_{k\rightarrow \infty }\sup_{t\leq T}|
r_{1}(X_{t}^{(n_{k})}(x))^{-1}-r_{1}(X_{t}(x))^{-1}|_0=0.
\end{equation}
Fix an $\omega $ for which both limits are zero. We will prove that 
\begin{equation}\label{eq:inversesubseqconvergefixedomega}
\lim_{k\rightarrow \infty }\sup_{t\leq T}|X_{t}^{(n_{k});-1}(\omega )-X_{t}^{-1}(\omega )|_{0;\left\{ \left\vert x\right\vert \leq R\right\} } 
=0.  
\end{equation}%
Suppose, by contradiction, that (\ref{eq:inversesubseqconvergefixedomega}) is
not true. Then there exists an $\varepsilon >0$ and a subsequence of $(n_{k})
$ (still denoted $(n_{k}))$ such that $t_{n_{k}}\rightarrow t-$ (or $%
t_{n_{k}}\rightarrow t+$) and $x_{n_{k}}\rightarrow x$ as $k\rightarrow \infty $
with $\left\vert x_{n_{k}}\right\vert \leq R$ such that (dropping $\omega )$,%
\begin{equation}\label{nf12}
|
X_{t_{n_{k}}}^{(n_{k});-1}(x_{n_{k}})-X_{t_{n_{k}}}^{-1}(x_{n_{k}})| \geq \varepsilon .  
\end{equation}%
Arguing by contradiction and using (\ref{eq:oneoverlimit}), we have  
\begin{equation}\label{ineq:boundsequence}
\sup_{k}|X_{t_{n_{k}}}^{(n_{k});-1}(x_{n_{k}})|<\infty .
\end{equation}
Applying \eqref{ineq:boundsequence}, \eqref{eq:limitdirectRbar}, and the fact that $X_{\cdot}(\cdot),X^{-1}_{\cdot}(\cdot)\in D([0,T];\cC_{loc}^{\beta'}(\bR^d;\bR^d))$ , we obtain
\begin{gather*}
\lim_{k\rightarrow \infty
}\left(X_{t-}(X_{t_{n_{k}}}^{(n_{k});-1}(x_{n_{k}}))-X_{t-}(X_{t_{n_{k}}}^{-1}(x_{n_{k}}))\right)=\lim_{k\rightarrow \infty
}\left(X_{t-}(X_{t_{n_{k}}}^{(n_{k});-1}(x_{n_{k}}))-x_{n_{k}}\right)\\
=\lim_{k\rightarrow \infty
}\left(X_{t-}(X_{t_{n_{k}}}^{(n_{k});-1}(x_{n_{k}}))-X_{t_{n_{k}}}^{(n_{k})}(X_{t_{n_{k}}}^{(n_{k});-1}(x_{n_{k}}))\right)
\end{gather*}
\begin{gather*}
=\lim_{k\rightarrow \infty
}\left(X_{t-}(X_{t_{n_{k}}}^{(n_{k});-1}(x_{n_{k}}))-X_{t_{n_{k}}}(X_{t_{n_{k}}}^{(n_{k});-1}(x_{n_{k}}))\right)\\
+\lim_{k\rightarrow \infty
}\left(X_{t_{n_{k}}}(X_{t_{n_{k}}}^{(n_{k});-1}(x_{n_{k}}))-X_{t_{n_{k}}}^{(n_{k})}(X_{t_{n_{k}}}^{(n_{k});-1}(x_{n_{k}}))\right)
=0,
\end{gather*}
which contradicts \eqref{nf12}, and hence proves \eqref
{eq:inversesubseqconvergefixedomega},  \eqref{nf10}, and  \eqref{eq:pointwiseinverselimit}.
For each $n$, define 
$$\bar{U}^{(n)}_t=\bar{U}^{(n)}(t,x)=\nabla
X_{t}^{(n)}(x) ^{-1}\quad \textrm{and}\quad \bar{U}(t)=\bar{U}(t,x)
=\nabla X_{t}(x)^{-1},\;\;(t,x)\in [0,T]\times \mathbf{R}^{d}.$$  Using \eqref{ineq:momentestgradinvbd} and  \eqref{ineq:momentestgradinvdiff} and repeating the arguments given above, for each  $p\geq 2,$ we get
\begin{equation}\label{f3}
\lim_{n}\bE[\sup_{t\leq T}|r_{1}^{-\epsilon }\bar{U}^{(n)}_t-%
r_{1}^{-\epsilon }\bar{U}_t|_{\beta'-1}^{p}]=0.  
\end{equation}
Then (\ref{f3}) and (\ref{nf10}) imply that for each $R>0$, 
\begin{gather*}
d\bP-\lim_{n\rightarrow\infty}\sup_{t\leq T}|\nabla
X_{t}^{(n);-1}(x)-\nabla X_{t}^{-1}(x)|_{0;\left\{ \left\vert x\right\vert \leq R\right\} }\\
=d\bP-\lim_{n\rightarrow\infty}\sup_{t\leq T}|\nabla
X_{t}^{(n)}(X_{t}^{(n);-1}(x))^{-1}-\nabla X_{t}(X_{t}^{-1}(x))^{-1}|_{0;\left\{ \left\vert x\right\vert \leq R\right\} }=0,
\end{gather*}
which yields \eqref{eq:pointwisegradientlimit} and completes the proof.
\end{proof}

\section{Classical solution of  an SPDE: Proof of Theorem \ref{thm:SPDEEx}}\label{sec:classicalsolutionctsexist}

\begin{proof}[Proof of Theorem \ref{thm:SPDEEx}]
Fix a stopping time $\tau\le T$.  By virtue of Theorem \ref{thm:diffeoandmomest}, we
only need to show that $Y^{-1}(\tau)=Y^{-1}_t(\tau,x)$ solves \eqref{eq:SPDEIntro} and that this is the unique solution.
Suppose we have shown $Y^{-1}(s,x)$, $s\in [0,T]$, solves \eqref{eq:SPDEIntro} (i.e. where $\tau$ is deterministic). 
It is then straightforward to conclude that $Y^{-1}(\tau')$ solves \eqref{eq:SPDEIntro}  for a  finite-valued stopping times $\tau'$.  We can then use an approximation argument (see the proof of Proposition \ref{prop:homeomorphism}) to show that $Y^{-1}(\tau)=Y^{-1}_t(\tau,x)$ solves \eqref{eq:SPDEIntro}. Thus, it suffices to take  $\tau$ deterministic.  Let $
u_t(x)$ $=u_t(s,x)$ $=Y_t^{-1}(s,x)$, $(s,t,x)\in [0,T]^2\times\mathbf{R}^d$. Fix  $(s,t,x)\in [0,T]^2\times\mathbf{R}^d$ with $s<t$ and write $Y_t(x)=Y_t(s,x)$.  We will treat a general stopping time $\tau\le T$ later.  Let $((t^M_n)_{0\le n\le M})_{1\le M\le \infty}$ be a sequence of partitions of the interval $%
[s,t]$ such that for each $M>0$, $(t^M_n)_{0\le n\le M}$ has mesh size $%
(t-s)/M$.  Fix $M$ and set $(t_n)_{0\le n\le
M}=(t^M_n)_{0\le n\le M}$. Immediately, we obtain
 \begin{equation}  \label{eq:Telescoping Sum}
u_t(x)-x=\sum_{n=0}^{M-1}(u_{t_{n+1}}(x)-u_{t_{n}}(x)).
\end{equation}%
We will use Taylor's theorem to expand each term in the sum on the right-hand-side of \eqref{eq:Telescoping Sum}. By Taylor's theorem, for each $n$ and $y$, we have
\begin{gather}
u_{t_{n+1}}(Y_{t_{n+1}}(y))-u_{t_n}(Y_{t_{n+1}}(y))=y-u_{t_n}(Y_{t_{n+1}}(y))=u_{t_n}(Y_{t_n}(y))-u_{t_n}(Y_{t_{n+1}}(y))\notag \\  \label{eq:Taylor for u at X}
=\nabla  u_{t_n}(Y_{t_n}(y))(Y_{t_n}(y)-Y_{t_{n+1}}(y))-(Y_{t_n}(y)-Y_{t_{n+1}}(y))^*\Theta_n(Y_{t_n}(y))(
Y_{t_n}(y)-Y_{t_{n+1}}(y)),
\end{gather}%
where 
$$
\Theta_n^{ij}(z)=\int_{0}^{1}(1-\theta )\partial_{ij}u_{t_n}\left(z+\theta
(Y_{t_{n+1}}(Y_{t_n}^{-1}(z))-z)\right)d\theta .
$$
Since for each $n$,
$
Y_{t_{n+1}}(s,x)=Y_{t_{n+1}}(t_n,Y_{t_n}(s,x)),
$
we have
$$
Y_{t_{n+1}}(Y_{t_{n}}^{-1}( x))=Y_{t_{n+1}}(t_{n},x)
$$
and hence substituting $y=Y_{t_{n}}^{-1}(x)$ into \eqref{eq:Taylor for u at X}, for each $n$, we get
\begin{equation}\label{eq:firstexpansion}
u_{t_{n+1}}(x)-u_{t_{n}}(x)=A_{n}+B_{n},
\end{equation}%
where 
$$
A_{n}:=\nabla u_{t_{n}}(x)(x-Y_{t_{n+1}}(t_{n},x))-(
x-Y_{t_{n+1}}(t_{n},x))^*\Theta^{ij}_n(x)(
x-Y_{t_{n+1}}(t_{n},x))
$$
and
$$
B_{n}:= (u_{t_{n+1}}(x)-u_{t_{n}}(x)) -(
u_{t_{n+1}}(Y_{t_{n+1}}(t_{n},x))-u_{t_{n}}(Y_{t_{n+1}}(t_{n},x))).
$$
Applying Taylor's theorem once more, for each $n$, we obtain 
\begin{equation}\label{eq:secondexpansion}
B_{n}=C_{n}+D_{n},
\end{equation}%
where 
$$
C_{n}:=(\nabla u_{t_{n+1}}(x)-\nabla u_{t_{n}}(x))(x-Y_{t_{n+1}}(t_{n},x)),
$$
$$
D_{n}:=-( x-Y_{t_{n+1}}(t_{n},x))^*\tilde{\Theta}_{n}(x)(
x-Y_{t_{i+1}}(t_{i},x))) ,
$$
and 
$$
\tilde{\Theta}_{n}(x)^{ij}:=\int_{0}^{1}(1-\theta )\partial_{ij}(u_{t_{n+1}}-u_{t_{n}})(x+\theta
(Y_{t_{n+1}}(t_{n},x)-x))d\theta.
$$
Thus, combining \eqref{eq:Telescoping Sum}, \eqref{eq:firstexpansion}, and \eqref{eq:secondexpansion},   $\bP$-a.s.\ we have
\begin{equation}\label{eq:expansion of u}
u_t(x)-x=
\sum_{n=0}^{M-1}(A_n+C_n+D_n).
\end{equation}%
Now, we will derive the limit  of the  right-hand-side of \eqref{eq:expansion of u}. 
\begin{claim}
\label{claim:Continuous SDE Limiting Procedure}
\begin{tightenumerate}
\item 
\begin{align*}
d\bP-\lim_{M\rightarrow\infty} \sum_{n=0}^{M-1}A_{n}&
=-\int_{]s,t]}[\frac{1}{2}\sigma ^{i\varrho}_r(x)
\sigma_r ^{j\varrho}(x)\partial_{ij}u_r(x)+b^{i}_r(x)\partial_iu_r(x)]dr \\
&\quad -\int_{]s,t]}\sigma_r^{i\varrho}
(x)\partial_iu_r(x)dw^{\varrho}_{r};
\end{align*}
\item$d\bP-\lim_{M\rightarrow\infty}\sum_{n=0}^{M-1}D_{n}=0;$
\item 
$
d\bP-\lim_{M\rightarrow\infty}\sum_{n=0}^{M-1}C_{n}=\int_{]s,t]}\sigma
^{j\varrho}_r(x)\partial_j \sigma^{i\varrho}_r(x)\partial_iu_r(x)dr+\int_{]s,t]}\sigma
^{i\varrho}_r(x) \sigma ^{j\varrho}_r(x)\partial_{ij}u_r(x)dr.
$
\end{tightenumerate}
\end{claim}
\begin{proof}[Proof of Claim \protect\ref{claim:Continuous SDE Limiting
Procedure}]
(1) For each $n$, we have%
\begin{align*}
\nabla u_{t_{n}}(x)\left( x-Y_{t_{n+1}}(t_{n},x)\right)& =-\int_{]t_{n},t_{n+1}]}b^i_r(x)
\partial_i u_{t_{n}}(x)dr-\int_{]t_{n},t_{n+1}]}\sigma^{i\varrho}_r (x)\partial_i u_{t_{n}}(x)dw^{\varrho}_{r}\\&\quad +R^{(1)}_n+R^{(2)}_n,
\end{align*}%
where 
\begin{equation*}
R^{(1)}_n:=\int_{]t_{n},t_{n+1}]}\left(b^i_r(x)-b^i_r(Y_{r}(t_{n},x))\right)\partial_i u_{t_{n}}(x)dr
\end{equation*}
and
\begin{equation*}
R^{(2)}_n:=\int_{]t_{n},t_{n+1}]}[\sigma_r^{i\varrho}
(x)-\sigma^{i\varrho}_r (Y_{r}(t_{n},x))]\partial_i u_{t_{n}}(x)dw^{\varrho}_{r}.
\end{equation*}
Since $b$ and $\sigma$ are Lipschitz, there is a constant $%
N=N(N_{0},T)$ such that 
$$
\sum_{n=0}^{M-1}\left \vert R^{1}_n\right\vert \leq N\sup_{s \le
r\leq t}|\nabla u_r(x)|\sup_{|r_1-r_2|\leq \frac{t}{M}}|x-Y_{r_1}(r_2,x)|
$$
and 
$$
\int_{]s,t]}\left\vert\sum_{n=0}^{M-1}\mathbf{1}
_{]t_{n},t_{n+1}]}(r)\left(\sigma_r^{i\cdot}
(x)-\sigma^{i\cdot}_r (Y_{r}(t_{n},x))\right)\partial_i u_{t_{n}}(x)\right\vert ^{2}ds
$$
$$
\leq N\sup_{s \le r\leq t}|\nabla u_r(x)|^2\sup_{|r_1-r_2|\leq \frac{t}{ M}%
}|x-Y_{r_1}(r_2,x)|^{2}.
$$
Owing to the joint continuity of $Y_t(s,x)$ in $s$ and $t$ and the dominated
convergence theorem for stochastic integrals, we obtain
\begin{equation}\label{eq:limitofremainder1}
d\bP-\lim_{M\rightarrow\infty}\sum_{n=0}^{M-1}(R^{(1)}_n+R^{(2)}_n)=0.
\end{equation}
In a similar way, this time using the continuity of $\nabla u_t(x)$ in $t$ and the linear growth of $b$ and $\sigma$, we get
$$
d\bP-\lim_{M\rightarrow\infty}\sum_{n=0}^{M-1}\left(-\int_{]t_{n},t_{n+1}]}b^i_r(x)
\partial_i u_{t_{n}}(x)dr-\int_{]t_{n},t_{n+1}]}\sigma^{i\varrho}_r (x)\partial_i u_{t_{n}}(x)dw^{\varrho}_{r}\right)
$$
$$
=-\int_{]s,t]}b_r(x)\partial_iu_r(x)dr-\int_{]s,t]}\sigma^{\varrho}
_r(x)\partial_iu_r(x)dw_{r}^{\varrho}.
$$
For each $n$, we have
\begin{equation*}
-(x-Y_{t_{n+1}}(t_{n},x))^* \Theta_n(x)(
x-Y_{t_{n+1}}(t_{n},x))=S^{(1)}_n+S^{(2)}_n,
\end{equation*}%
where $S^{(1)}_n(t,x)$ has only $drdr$ and $drdw_r^{\varrho}$ terms and where
\begin{align*}
S^{(2)}_n:&=-\frac{1}{2}\left( \int_{]t_{n},t_{n+1}]}\sigma^{i\varrho}_r
(Y_{r}(t_{n},x))dw_{r}^{\varrho}\right)\partial_{ij}u_{t_{n}}(x)\left(\int_{]t_{n},t_{n+1}]}\sigma^{j\varrho}_r
(Y_{r}(t_{n},x))dw_{r}^{\varrho}\right)\\
&\quad -\left(\int_{]t_{n},t_{n+1}]}\sigma^{i\varrho}_r (Y_{r}(t_{n},x))dw^{\varrho}_{r}\right)
\left(\Theta^{ij}_n(x)-\frac{1}{2}\partial_{ij}u_{t_{n}}(x)\right)
\left(\int_{]t_{n},t_{n+1}]}\sigma^{j\varrho}_r (Y_{r}(t_{n},x))dw^{\varrho}_{r}\right).
\end{align*}
Since
\begin{align*}
\left\vert \Theta^{ij}_n(x)-\frac{1}{2}
\partial_{ij}u_{t_{n}}(x)\right\vert
&=\left\vert \int_{0}^{1}(1-\theta )(\partial_{ij}u_{t_{n}}(x+\theta
(Y_{t_{n+1}}(t_{n},x)-x))-\partial_{ij}u_{t_{n}}(x))d\theta \right\vert\\
&\leq N\sup_{|r_1-r_2|\leq \frac{t}{M},\theta \in
(0,1)}|\partial_{ij}u_{r_{1}}(x+\theta
(Y_{r_2}(r_1,x)-x))-\partial_{ij}u_{r_{1}}(x))|,
\end{align*}
proceeding as in the derivation of \eqref{eq:limitofremainder1} and using   the joint continuity of $\partial_{ij}u_t(x)$ in $t$ and $x$,  the continuity of $Y_t(s,x)$ in $s$ and $t$, and standard properties of the stochastic
integral (i.e. Thm. 2 (5) in \cite{LiSh89} and the stochastic dominated convergence theorem), we obtain
\begin{equation*}
d\bP-\lim_{M\rightarrow\infty}\sum_{n=0}^{M-1} S^{(2)}_n =-\frac{1}{2}%
\int_{]0,t]}\sigma ^{i\varrho}_r(x) \sigma ^{j\varrho}_r(x)\partial_{ij}u_r(x)dr.
\end{equation*}
Similarly, by appealing to standard
properties of the stochastic integral and the properties stated in  Theorem \ref{thm:diffeoandmomest}(2), we have
$
d\bP-\lim_{M\rightarrow\infty}\sum_{n=0}^{M-1}S^{(1)}_n=0,
$
which completes the proof of part (1). The proof of part (2) is  similar to the proof of part (1), so we proceed to
the proof of  part (3). We know that for each $n$,
$
Y_{t_{n+1}}(x)=Y_{t_{n+1}}(t_n,Y_{t_n}(x)).
$
Thus, for each $n$, we have
$
u_{t_{n+1}}(x)=u_{t_n}(Y^{-1}_{t_{n+1}}(t_n,x)),
$
and hence by the chain rule,
\begin{equation}\label{eq:Inverse identity for Du}
\nabla  u_{t_{n+1}}(x)=\nabla  u_{t_n}(Y_{t_{n+1}}^{-1}(t_n,x))\nabla 
Y_{t_{n+1}}^{-1}(t_n,x).
\end{equation}%
By \eqref{eq:Inverse identity for Du} and Taylor's theorem, for each $n$, we get
\begin{gather*}
C_{n}=(\nabla  u_{t_{n+1}}(x)-\nabla u_{t_{n}}(x))(x-Y_{t_{n+1}}(t_{n},x))\\
=\nabla u_{t_n}(Y_{t_{n+1}}^{-1}(t_n,x))(\nabla 
Y_{t_{n+1}}^{-1}(t_n,x)-I_d)(x-Y_{t_{n+1}}(t_n,x))\\
+(Y_{t_{n+1}}^{-1}(t_n,x)-x)^*\tilde{\Theta}_n(x)(x-Y_{t_{n+1}}(t_n,x))=:E_{n}+F_{n},
\end{gather*}
where 
$$\tilde{\Theta}^{ij}_n(x):=\int_0^1 \partial_{ij}u_{t_n}(x+\theta(Y_{t_{n+1}}^{-1}(t_n,x)-x))d\theta.$$
By It\^{o}'s formula, for each $n$, we have (see, also, Lemma 3.12 in \cite{Ku04}),
\begin{align*}
\nabla Y_{t_{n+1}}(t_n,x)^{-1}&=I_d-\int_{]t_n,t_{n+1}]} \nabla Y_{r}(t_n,x)^{-1}\nabla \sigma_r^{\varrho}
  (Y_{r}(t_n,x))dw^{\varrho}_{r}\\
&\quad +\int_{]t_n,t_{n+1}]}\nabla Y_{r}(t_n,x)^{-1}\left(\nabla
  \sigma^{\varrho}_r(Y_{r}(t_n,y)) \nabla \sigma^{\varrho}_r (Y_{r}(t_n,x))-\nabla b_r(Y_{r}(t_n,x))\right)dr,
\end{align*}
and hence
 $$ \nabla Y_{t_{n+1}}^{-1}(t_n)-I_d=\nabla Y_{t_{n+1}}^{-1}(t_n,Y^{-1}_{t_{n+1}}(t_n,x))-I_{d}=:G_{t_n,t_{n+1}}^{(1)}(Y^{-1}_{t_{n+1}}(t_n,x))+G_{t_n,t_{n+1}}^{(2)}(Y^{-1}_{t_{n+1}}(t_n,x)),
 $$
where for $y\in\mathbf{R}^d$,
 $$
 G_{t_n,t_{n+1}}^{(1)}(y):=\int_{]t_n,t_{n+1}]} \nabla Y_{r}(t_n,z)^{-1}\left(\nabla
   \sigma_r^{\varrho}(Y_{r}(t_n,y)) \nabla \sigma_r ^{\varrho}(Y_{r}(t_n,y))-\nabla b_r(Y_{z}(t_n,y))\right)dr
 $$
 and
 $$
  G_{t_n,t_{n+1}}^{(2)}(z):=-\int_{]t_n,t_{n+1}]} \nabla Y_{r}(t_n,y)^{-1}\nabla \sigma_r^{\varrho}
  (Y_{r}(t_n,y))dw^{\varrho}_{r}.
 $$
 By the Burkholder-Davis-Gundy inequality, H\"{o}lder's inequality, and  the inequalities  \eqref{ineq:estdirectdifftposp}, \eqref{ineq:momentestgradinvbd}, and \eqref{ineq:momentestgradinvdiff}, for  each $p\geq 2$, there
 is a constant $N=N(p,d,N_{0},T)$ such that for all $x_1$ and $x_2$, 
 $$
 \bE\left[|G_{t_n,t_{n+1}}^{(2)}(x_1)|^p\right]\le NM^{-p/2+1}\int_{]t_n,t_{n+1}]}\bE\left[| \nabla Y_{r}(t_n,x_1)^{-1}|^p|\nabla \sigma_r
   (Y_{r}(t_n,x_1)|^p\right]dr\le N M^{-p/2}
 $$
 and
\begin{gather*}
 \bE\left[|G_{t_n,t_{n+1}}^{(2)}(x_1)-G_{t_n,t_{n+1}}^{(2)}(x_2)|^{p}\right]
 \leq N M^{-p/2+1}\int_{]t_n,t_{n+1}]}\bE\left[|\nabla Y_{r}(t_n,x_1)^{-1}-\nabla Y_{r}(t_n,x_2)^{-1}|^{p}\right]dr\\
 + NM^{-p/2+1}\int_{]t_n,t_{n+1}]}\left( \bE\left[|\nabla Y_{r}(t_n,x_1)^{-1}|^{2p}\right]\right)
 ^{1/2}\left( \bE\left[|Y_{r}(t_n,x_1)-Y_{r}(t_n,x_2)|^{2p}\right]\right) ^{1/2}dr\\
 \leq NM^{-p/2}|x-y|^{p}.
\end{gather*}
 Thus, by Corollary  \ref{cor:Kolmogorov Embedding}, we obtain that for all $p\ge 2$, $\epsilon>0$, and  $\delta <1$, there is a constant $N=N(p,d,\delta, N_{0},T)$ such that 
 \begin{equation}\label{ineq:Kolmogorov Theorem for SI Terms}
 \bE\left[|r^{-\epsilon}G_{t_n,t_{n+1}}^{(2)}|_{\delta }^{p}\right]\leq NM^{-p/2}.
 \end{equation}
For each $n$, we have 
  \begin{align*}
  E_{n}&=\nabla u_{t_n}(Y_{t_{n+1}}^{-1}(t_n,x))G_{t_n,t_{n+1}}^{(1)}(Y_{t_{n+1}}^{-1}(t_n,x))(x-Y_{t_{n+1}}(t_n,x))\\
  &\quad + \nabla u_{t_n}(x)
               G_{t_n,t_{n+1}}^{(2)}(x)
               (x-Y_{t_{n+1}}(t_n,x))\\
&\quad +\nabla u_{t_n}(Y_{t_{n+1}}^{-1}(t_n,x))(
  G_{t_n,t_{n+1}}^{(2)}(Y_{t_{n+1}}^{-1}(t_n,x))-G_{t_n,t_{n+1}}^{(2)}(x)%
  ) (x-Y_{t_{n+1}}(t_n,x))\\
   &\quad +  (\nabla u_{t_n}(Y_{t_{n+1}}^{-1}(t_n,x)-\nabla u_{t_n}(x))
        G_{t_n,t_{n+1}}^{(2)}(x)%
        (x-Y_{t_{n+1}}(t_n,x))
  \end{align*}
One can easily check that 
  \begin{equation}\label{eq:part1K} 
 d\bP-\lim_{M\rightarrow\infty}\sum_{n=0}^{M-1}\nabla u_{t_n}(Y_{t_{n+1}}^{-1}(t_n,x))G_{t_n,t_{n+1}}^{(1)}(Y_{t_{n+1}}^{-1}(t_n,x))(x-Y_{t_{n+1}}(t_n,x))=0.
  \end{equation}
Since $\nabla u_t(x)$ is jointly continuous in $t$ and $x$ and $Y^{-1}_t(s,x)$ is jointly in $s$ and $t$, 
we have
$$
 d\bP-\lim_{M\rightarrow\infty}\sup_n|\nabla u_{t_n^M}(Y_{t_{n+1}^M}^{-1}(t_n^M,x))-\nabla u_{t_n^M}(x)|=0.
$$
Moreover, using  H{\"o}lder's inequality, \eqref{ineq:Kolmogorov Theorem for SI Terms}, and \eqref{ineq:growthdirectposp}, we get
  $$
 \sup_M  \bE\sum_{n=0}^{M-1}|G_{t_n,t_{n+1}}^{(2)}(x)|x-Y_{t_{n+1}}(t_n,x)|<\infty,
  $$
  and hence
  \begin{equation}\label{eq:part2K} 
 d\bP-\lim_{M\rightarrow\infty}\sum_{n=0}^{M-1}(\nabla u_{t_n}(Y_{t_{n+1}}^{-1}(t_n,x))-\nabla u_{t_n}(x))
         G_{t_n,t_{n+1}}^{(2)}(x)%
         (x-Y_{t_{n+1}}(t_n,x))=0.
  \end{equation}
 We claim that  
 \begin{equation}\label{eq:part3K}
d\bP-\lim_{M\rightarrow\infty}\sum_{n=0}^{M-1}\nabla u_{t_n}(Y_{t_{n+1}}^{-1}(t_n,x))\left(
 G_{t_n,t_{n+1}}^{(2)}(Y_{t_{n+1}}^{-1}(t_n,x))-G_{t_n,t_{n+1}}^{(2)}(x)%
 \right) (x-Y_{t_{n+1}}(t_n,x))=0.
 \end{equation}
Set 
\begin{equation*}
J^M=\sum_{n=0}^{M-1}|
G_{t_n,t_{n+1}}^{(2)}(Y_{t_{n+1}}^{-1}(t_n,x))-G_{t_n,t_{n+1}}^{(2)}(x)|x-Y_{t_{n+1}}(t_n,x)|.
\end{equation*}
For each $\bar\delta ,\epsilon \in (0,1)$, we have 
$$
\bP(J^M>\bar \delta )\leq \bP\left(J^{M}>\bar \delta
,\;\max_{n }|Y_{t_{n+1}}^{-1}(t_n,x)-x|\leq \epsilon \right)+\bP\left(\max_{n}|Y_{t_{n+1}}^{-1}(t_n,x)-x|>\epsilon \right).
$$
By virtue of \eqref{ineq:Kolmogorov Theorem for SI Terms},    there is a deterministic constant $N=N(x)$ independent of $M$ such that  for all $\omega\in 
V^{M}:=\{\max_{n}|Y_{t_{n+1}}^{-1}(t_n,x)-x|\leq \epsilon \}$,
\begin{equation*}
J^{M}\leq N \epsilon ^{\delta }\sum_{n=0}^{M-1}[r_1^{-\epsilon}G_{t_n,t_{n+1}}^{(2)}]_{\delta}|x-Y_{t_{n+1}}(t_n,x)|,
\end{equation*}%
which implies that 
\begin{equation*}
\bE\mathbf{1} _{V^M}J^{M}\leq  N\epsilon ^{\delta
}\bE\sum_{n=0}^{M-1}\left([r_1^{-\epsilon}G_{t_n,t_{n+1}}^{(2)}]_{\delta}^{2}+
|x-Y_{t_{n+1}}(t_n,x)|^{2}\right)
\leq N\epsilon ^{\delta }\sum_{n=0}^{M-1}M^{-1}\leq N\epsilon ^{\delta }.
\end{equation*}%
Applying Markov's inequality, we get 
\begin{equation*}
\bP(J^{M}|>\bar \delta ,\;\max_{n}|Y_{t_{n+1}}^{-1}(t_n,x)-x|\leq
\epsilon )\leq N\frac{\epsilon ^{\delta }}{\bar\delta },
\end{equation*}%
and hence for all $\bar \delta>0$,
\begin{equation}\label{eq:Jmconverges}
\lim_{M\rightarrow \infty }\bP(J^{M}>\bar \delta )=0,
\end{equation}%
which yields \eqref{eq:part3K}.
Owing to \eqref{eq:part1K}, \eqref{eq:part2K}, and \eqref{eq:part3K}, we have
\begin{equation*}
d\bP-\lim_{M\rightarrow\infty}\sum_{n=0}^{M-1}E_{n}=\lim_{M\rightarrow
\infty }\sum_{n=0}^{M-1}\nabla
u_{t_n}(x)G_{t_n,t_{n+1}}^{(2)}(x)(x-Y_{t_{n+1}}(t_n,x)).
\end{equation*}%
Proceeding as in the proof of part (1) of the claim, we obtain 
\begin{gather}
\lim_{M\rightarrow \infty }\sum_{n=0}^{M-1}K_{n}\\
=\lim_{M\rightarrow \infty }\sum_{n=0}^{M-1}\nabla 
u_{t_n}(x)\int_{]t_n,t_{n+1}]}(\nabla Y_{r}(t_n,x)^{-1}-I_{d})\nabla \sigma^{\varrho}_r
(x)dW^{\varrho}_{r}\int_{]t_n,t_{n+1}]}\sigma^{\varrho}_r(x)dW^{\varrho}_{r}\\
+\lim_{M\rightarrow \infty }\sum_{n=0}^{M-1}\nabla
u_{t_n}(x)\int_{]t_n,t_{n+1}]}\nabla \sigma^{\varrho}_r(x)dw_{r}^{\varrho}\int_{]t_n,t_{n+1}]}\sigma^{\varrho}_r(x)dw^{\varrho}_{r}\\
\label{eq:limitofKn}
=\int_{]s,t]}\sigma
^{j\varrho}_r(x)\partial_j\sigma^{i\varrho}_r(x)\partial_iu_r(x)dr
\end{gather}
It is easy to check that for each $n$,
\begin{align*}
F_{n}&=(Y_{t_{n+1}}^{-1}(t_n,x)-x)^*\tilde{\Theta}_n(x)(x-Y_{t_{n+1}}(t_n,x))\\
&=:(G_{t_n,t_{n+1}}^{(3)}(Y^{-1}_{t_{n+1}}(t_n,x))+G_{t_n,t_{n+1}}^{(4)}(Y^{-1}_{t_{n+1}}(t_n,x)))^*\tilde{\Theta}_n(x)(x-Y_{t_{n+1}}(t_n,x)),
\end{align*}
where for $y\in\mathbf{R}^d$,
  $$
  G_{t_n,t_{n+1}}^{(3)}(y):=-\int_{]t_n,t_{n+1}]}b_r (Y_{r}(t_n,y))dr,\quad  G_{t_n,t_{n+1}}^{(4)}(y):=-\int_{]t_n,t_{n+1}]}\sigma^{\varrho}_r (Y_{r}(t_n,y))dw^{\varrho}_{r}.
  $$
Arguing as  in the proof of \eqref{eq:limitofKn}, we get
\begin{equation*}
d\bP-\lim_{M\rightarrow\infty}\sum_{n=0}^{M-1}F_{n}=\int_{]s,t]}\sigma ^{i\varrho}_r(x) \sigma ^{j\varrho}_r(x)\partial_{ij}u_r(x)dt,
\end{equation*}%
which completes the proof of the claim.
\end{proof}
By virtue of  \eqref{eq:expansion of u} and Claim \ref{claim:Continuous SDE Limiting Procedure}, for all $s$ and $t$ with $s\le t$ and $x$, $\bP$-a.s.\ 
\begin{equation}\label{eq:proofctspde}
u_t(x)=x+\int_{]s,t]} \left( \frac{1}{2}\sigma ^{i\varrho}_r(x) \sigma
^{j\varrho}(x)\partial_{ij}u_r(x)-\hat b^i_t(x)\partial_iu_r(x)\right)dr-\int_{]s,t]}\sigma ^{i\varrho}_r(x)\partial_iu_r(x)dw^{\varrho}_{r}.
\end{equation}
Owing to Theorem \ref{thm:diffeoandmomest},  $u=u_t(x)$ has a
modification that is jointly continuous in $s$ and $t $ and twice continuously
differentiable in $x$.  It is easy to check that the Lebesgue integral  on the right-hand-side of \eqref{eq:proofctspde}
has a modification that is continuous in $s$, $t$, and $x$.   Thus, the stochastic integral on the right-hand-side of  \eqref{eq:proofctspde} has a modification that
is continuous in $s,t,$ and $x$, and hence the equality in \eqref{eq:proofctspde}
holds $\bP$-a.s.\ for all  $s$ and $t$ with $s\le t$ and $x$.  This proves that   $Y^{-1}(\tau)=Y^{-1}_t(\tau,x)$ solves \eqref{eq:SPDEIntro}. However, if $u^1(\tau), u^2(\tau)\in \mathfrak{C}^{\beta'}_{cts}(\bR^d;\bR^d) $ are  solutions of \eqref{eq:SPDEIntro}, then applying the It\^o-Wentzell formula (see, e.g. Theorem 9 in Chapter 1, Section 4.8 in \cite{Ro90}), we get that $\bP$-a.s.\ for all $t$ and $x$, $$u^1_t(\tau,Y_t(\tau,x))=x=u_t^2(\tau,Y_t(\tau,x)),$$
which implies  that $\bP$-a.s.\ for all $t$ and $x$, $u^1(\tau)=Y_t^{-1}(\tau,x)=u^2(\tau)$.  Thus, $Y^{-1}(\tau)=Y^{-1}_t(\tau,x)$ is the unique solution of  \eqref{eq:SPDEIntro} in $\mathfrak{C}^{\beta'}_{cts}(\bR^d;\bR^d).$
\end{proof}

\section{Appendix}
Let $V$ be an arbitrary Banach space.  The following lemma and its corollaries are indispensable in this paper.

\begin{lemma}\label{lem:Kembedding}
Let $Q\subseteq \mathbf{R}^{d}$ be
an open bounded cube, $p\geq 1$,  $\delta \in (0,1]$, \ and $f$ be a
$V$-valued integrable function on $Q$ such that 
\begin{equation*}
\left[ f\right] _{\delta;p;Q;V}:=\left( \int_{Q}\int_{Q}\frac{%
|f(x)-f(y)|_{V}^{p}}{|x-y|^{2d+\delta p}}dxdy\right) ^{1/p}<\infty .
\end{equation*}%
Then $f$ has a $\cC^{\delta }(Q;V)$-modification and there is a constant $%
N=N(d,\delta,p )$ independent of $f$ and $Q$ such that%
\begin{equation*}
[f]_{\delta ;Q;V}\leq N\left[ f\right] _{\delta ,p;Q;V}
\end{equation*}%
and
\begin{equation*}
\sup_{x\in Q}|f(x)| _{V}\leq N| Q|
^{\delta /d}[f] _{\delta;p;Q;V}+|Q|
^{-1/p}\left( \int_{Q}|f(x)| _{V}^{p}dx\right) ^{1/p},
\end{equation*}%
where $|Q|$ is the volume of the cube.  
\end{lemma}
\begin{proof}
If $V=\mathbf{R}$, then the existence of a continuous modification of $f$
and the estimate of $\left[ f\right] _{\delta ;Q}$ follows from Lemma 
2  and Exercise 5 in Section 10.1 in \cite{Kr08}. The proof for a
general Banach space is the same. For all $x\in Q$, we have 
\begin{align*}
|f(x)|_{V} &\leq \frac{1}{|Q|}\int_{Q}|
f(x)-f(y)|_{V}dy+\frac{1}{|Q|}%
\int_{Q}| f(y)|_{V}dy \\
&\leq N\frac{1}{|Q|}\left[ f\right] _{\delta
,p;Q}\int_{Q}|x-y|^{\delta }dy+\frac{1}{|
Q|}\int_{Q}| f(y)|_V dy \\
&\leq N|Q|^{\delta /d}[f] _{\delta
,p;Q}+| Q|^{-1/p}\left( \int_{Q}|
f(y)| _{V}^{p}dy\right) ^{1/p},\end{align*}%
which proves the second estimate.
\end{proof}
The following is a direct consequence of Lemma \ref{lem:Kembedding}.
\begin{corollary}\label{cor:SobolevFull}Let $p\geq 1$, $\delta \in (0,1]$, and $f$ be a
$V$-valued function on $\mathbf{R}^d$ such that 
\begin{equation*}
|f| _{\delta ;p;V}:=\left(\int_{\mathbf{R}^d}|f(x)|^p_Vdx+ \int_{|x-y|<1}\frac{%
|f(x)-f(y)|_{V}^{p}}{|x-y|^{2d+\delta p}}dxdy\right) ^{1/p}<\infty .
\end{equation*}%
Then $f$ has a $\cC^{\delta }(\mathbf{R}^d;V)$-modification and there is a constant $%
N=N(d,\delta, p)$ independent of $f$ such that%
\begin{equation*}
| f| _{\delta;V}\leq N| f| _{\delta ;p;V}.
\end{equation*}%
\end{corollary}
\begin{corollary}
\label{cor:Kolmogorov Embedding} Let $X$ be a $V$-valued random field
defined on $\mathbf{R}^{d}$. Assume that for some $p\geq 1$, $l\ge 0,$ and $%
\beta \in (0,1]$ with $\beta p>d$ there is a constant $\bar{N}>0$ such
that for all $x,y\in \mathbf{R}^{d}$, 
\begin{equation}
\bE\left[|X(x)|_{V}^{p}\right]\leq \bar{N}r_{1}(x)^{lp}  \label{ineq:growthsob}
\end{equation}%
and 
\begin{equation}
\bE\left[|X(x)-X(y)|_{V}^{p}\right]\leq \bar{N}[r_{1}(x)^{lp}+r_1(y)^{lp}]|x-y|^{\beta p}.
\label{ineq:diffsob}
\end{equation}%
Then for any $\delta \in (0,\beta -\frac{d}{p})$ and $\epsilon >\frac{d}{p}$%
, there exists a $\cC^{\delta }(\mathbf{R}^{d};V)$-modification of $%
r_{1}^{-(l+\epsilon )}X$ and a constant $N=N(d,p,\delta ,\epsilon )$ such
that 
\begin{equation*}
\bE\left[|r_{1}^{-(l+\epsilon )}X|_{\delta }^{p}\right]\leq N\bar{N}.
\end{equation*}
\end{corollary}

\begin{proof}
Fix $\delta \in (0,\beta -\frac{d}{p})$ and $\epsilon >\frac{d}{p}$. Owing to %
\eqref{ineq:growthsob}, there is a constant $N=N(d,p,\bar{N}%
,\delta ,\epsilon )$ such that 
\begin{equation*}
\int_{\mathbf{R}^{d}}\bE\left[|r_{1}(x)^{-(l+\epsilon )}X(x)|_{V}^{p}\right]dx\leq 
\bar{N}\int_{\mathbf{R}^{d}}r_{1}(x)^{-p\epsilon }dx\leq N\bar N.
\end{equation*}%
By the mean value theorem, for each $x$ and $y$ and $\bar{p}\in \mathbf{R}$, we
have 
$$
|r_{1}(x)^{\bar{p}}-r_{1}(y)^{\bar{p}}|\leq |\bar{p}|(r_{1}(x)^{\bar{p}%
-1}+r_{1}(y)^{\bar{p}-1})|x-y|.  \label{ineq:weightest}
$$
Appealing to \eqref{ineq:diffsob} and \eqref{ineq:weightest}, we obtain that
there is a constant $N=N(d,p,\delta ,\epsilon )$ such that 
\begin{gather*}
\int_{|x-y|<1}\frac{\bE\left[|r_{1}(x)^{-(l+\epsilon
)}X(x)-r_{1}(y)^{-(l+\epsilon )}X(y)|_{V}^{p}\right]}{|x-y|^{2d+\delta p}}dxdy\\
\leq \bar{N}\int_{|x-y|<1}\frac{r_{1}(x)^{-p\epsilon }+r_{1}(y)^{-p\epsilon }}{%
|x-y|^{2d-(\beta-\delta)p}}dxdy+\bar{N}\int_{|x-y|<1}\frac{%
r_{1}(y)^{pl}|r_{1}(x)^{-(l+\epsilon )}-r_{1}(y)^{-(l+\epsilon )}|^{p}}{%
|x-y|^{2d+\delta p}}dxdy\\
\leq N\bar{N}+N\bar{N}\int_{|x-y|<1}\frac{r_{1}(x)^{-p(1+\epsilon
)}+r_{1}(y)^{-p(1+\epsilon )}}{|x-y|^{2d-(1-\delta)p}}dxdy\leq N\bar{N}.
\end{gather*}%
Therefore, $\bE[r_{1}^{-(l+\epsilon )}X]_{\delta ,p}^{p}\leq N\bar N$, and
hence, by Corollary \ref{cor:Kolmogorov Embedding}, $r_{1}^{-(l+\epsilon )}X$ has a  $\cC^{\delta }(\mathbf{R}^{d};V)$-modification and the estimate follows immediately.
\end{proof}
\begin{corollary}\label{cor:SobEmbLimit}Let $(X^{(n)})_{n\in\bN}$ be a sequence of $V$-valued random field defined on $%
\mathbf{R}^{d}$. Assume that for some $p\geq 1$, $l\geq 0$ and $\beta \in
(0,1],$ with $\beta p>d$ there is a constant $\bar N>0$ such that for all $%
x,y\in \mathbf{R}^{d}$ and $n\in\mathbf{N}$,
$$
\bE\left[|X^{(n)}(x)|_{V}^{p}\right]\leq \bar Nr_{1}(x)^{lp}
$$
and 
$$
\bE\left[|X^{(n)}(x)-X^{(n)}(y)|_{V}^{p}\right]\leq \bar N(r_{1}(x)^{lp}+r_1(y)^{lp})|x-y|^{\beta p}.
$$
Moreover, assume that for each  $x\in \mathbf{R}^{d},$ 
$
\lim_{n\rightarrow\infty}\bE\left[|X^{(n)}(x)|^{p}\right]= 0.
$
Then for any $\delta \in (0,\beta -\frac{d}{p})$ and $\epsilon >\frac{d}{p}$, 
\begin{equation*}
\lim_{n\rightarrow\infty}\bE\left[|r_{1}^{-(l+\epsilon )}X^{(n)}|_{\delta }^{p}\right]=0.
\end{equation*}
\end{corollary}
\begin{proof}
Fix $\delta \in (0,\beta -\frac{d}{p})$ and $\epsilon >\frac{d}{p}$.
Using  the Lebesgue dominated convergence theorem, we get
\begin{equation*}
\lim_{n\rightarrow\infty }\int_{\mathbf{R}^{d}}\bE\left[|r_{1}(x)^{-(l+\epsilon
)}X^{(n)}(x)|_{V}^{p}\right]dx=0,
\end{equation*}%
and therefore for each $\zeta \in (0,1)$,
\begin{equation*}
\lim_{n}\int_{\zeta<|x-y|<1 }\frac{\bE\left[|r_{1}(x)^{-(l+\epsilon
)}X_{n}(x)-r_{1}(y)^{-(l+\epsilon )}X_{n}(y)|_{V}^{p}\right]}{|x-y|^{2d+\delta p}}%
dxdy=0.
\end{equation*}%
Repeating the proof of Corollary \ref{cor:Kolmogorov Embedding}, we obtain that there is a constant $N$ such that 
\begin{gather*}
\int_{|x-y|\leq \zeta }\frac{\bE\left[|r_{1}(x)^{-(l+\epsilon
)}X^{(n)}(x)-r_{1}(y)^{-(l+\epsilon )}X^{(n)}(y)|_{V}^{p}\right]}{|x-y|^{2d+\delta p}}dxdy\\
\le \bar N\int_{|x-y|\leq \zeta }\frac{r_{1}(x)^{-p\epsilon }+r_{1}(y)^{-p\epsilon }}{%
|x-y|^{2d+(\delta -\beta )p}}dxdy+\bar N\int_{|x-y|\leq \zeta }\frac{%
r_{1}(x)^{-p(1+\epsilon )}+r_{1}(y)^{-p(1+\epsilon )}}{|x-y|^{2d+(\delta
-1)p}}dxdy\\
\le \bar N\zeta ^{\beta p-\delta p-d}.
\end{gather*}
Therefore,  $\lim_{n\rightarrow\infty}\bE\left[[r_{1}^{-(l+\epsilon )}X]_{\delta ,p}^{p}\right]=0$,
and  the statement follows.
\end{proof}

\bibliographystyle{alpha}
\bibliography{../../bibliography}

\end{document}